\documentclass[a4paper]{amsart}

\usepackage{amsmath, amssymb, amscd, amsthm, a4wide,verbatim,enumerate,mathtools}
\usepackage{hyperref}

\usepackage{tikz}
\usetikzlibrary{arrows}

\newtheorem{thm}{Theorem}[section]
\newtheorem{cor}[thm]{Corollary}
\newtheorem{lemma}[thm]{Lemma}
\newtheorem{prop}[thm]{Proposition}

\theoremstyle{definition}
\newtheorem{remark}[thm]{Remark}
\numberwithin{equation}{section}

\newcommand{\R}{\mathbb R}
\newcommand{\N}{\mathbb N}
\newcommand{\C}{\mathbb C}
\newcommand{\Z}{\mathbb Z}
\newcommand{\T}{\mathbb T}
\newcommand{\al}{\alpha}
\newcommand{\be}{\beta}
\newcommand{\ga}{\gamma}

\newcommand{\de}{\delta}
\newcommand{\De}{\Delta}

\newcommand{\si}{\sigma}
\newcommand{\te}{\theta}
\newcommand{\Te}{\Theta}

\newcommand{\la}{\lambda}
\newcommand{\La}{\Lambda}
\newcommand{\Ups}{\Upsilon}

\newcommand{\tensor}{\otimes}
\newcommand{\rphis}[5]{\,_{#1}\varphi_{#2}\!\left( \genfrac{.}{.}{0pt}{}{#3}{#4}
\,;#5 \right)}
\newcommand{\hf}{\frac{1}{2}}
\newcommand{\sgn}{\mathrm{sgn}}
\newcommand{\mhyphen}{\text{--}}

\begin{document}
\date{\today}
\title{Coupling coefficients for tensor product representations of quantum $\mathrm{SU}(2)$}
\author{Wolter Groenevelt}
\address{Technische Universiteit Delft, DIAM, PO Box 5031,
2600 GA Delft, the Netherlands}
\email{w.g.m.groenevelt@tudelft.nl}

\begin{abstract}
We study tensor products of infinite dimensional representations (not corepresentations) of the $\mathrm{SU}(2)$ quantum group. Eigenvectors of certain self-adjoint elements are obtained, and coupling coefficients between different eigenvectors are computed. The coupling coefficients can be considered as $q$-analogs of Bessel functions. As a results we obtain several $q$-integral identities involving $q$-hypergeometric orthogonal polynomials and $q$-Bessel-type functions.
\end{abstract}

\maketitle

\section{Introduction}

Many identities for special functions have a structure coming from representation theory of Lie or quantum algebras. In e.g.~\cite{VdJ}, \cite{KoeVdJ}, \cite{Groen2} summation and integral identities are derived from tensor products of irreducible $*$-representations of the Lie algebras $\mathfrak{su}(2)$, $\mathfrak{su}(1,1)$ and quantum versions of these Lie algebras. In this paper we consider representations of the $\mathrm{SU(2)}$ quantum group, i.e.~the quantized algebra $\mathcal A_q=\mathcal A_q(\mathrm{SU}(2))$ of functions on the Lie group $\mathrm{SU(2)}$. In the classical $q=1$ case the representation theory of this algebra is trivial, but in the quantum case the algebra has a class of infinite dimensional irreducible $*$-representations. We consider two- and threefold tensor products of these representations, and compute coupling coefficients between eigenvectors of certain self-adjoint elements in the algebra $\mathcal A_q$. This leads to several $q$-integral identities involving different classes of $q$-hypergeometric special functions. The identities we obtain can be viewed as connection formulas between orthogonal bases in Hilbert spaces of functions in two variables.

In \cite{Koo91} Koornwinder studies the tensor product of two infinite dimensional representations of $\mathcal A_q$. Using spectral analysis of a self-adjoint element $\ga\be \in \mathcal A_q$, the tensor product can be decomposed into irreducible representations. The element $\ga\be$ is a special case of the self-adjoint element $\rho_{\tau,\si} \in \mathcal A_q$, $\tau,\si \in \R$. The latter was introduced by Koornwinder in \cite{Koo93}, in which an explicit expression for the Haar-functional on the subalgebra generated by $\rho_{\tau,\si}$ is derived using Askey-Wilson polynomials. Koelink and Verding \cite{KoeVer} obtained the Haar-functional in a different way using spectral analysis of $\rho_{\tau,\si}$ in an infinite dimensional representation of $\mathcal A_q$. The latter approach uses eigenvectors which are given explicitly in terms of $q$-hypergeometric orthogonal polynomials. In this paper we combine the approaches from \cite{KoeVer} and \cite{Koo91} to derive several $\rho_{\tau,\si}$-eigenvectors for tensor product representations. Having different explicit eigenvectors, it may be expected to find `nice' explicit expressions for the coupling coefficients between them.

The coupling coefficients we obtain can be considered as $q$-analogs of Bessel functions. The most general ones in this paper are the $q$-Meixner functions \cite{GroenK}. It is shown that several known types of $q$-Bessel functions can be considered as limit cases of the $q$-Meixner funtions. This leads to the scheme of $q$-Bessel functions in Figure \ref{fig:qBessel}. For the definitions of the functions in this scheme and their Hankel-type orthogonality relations see Sections \ref{sec:twofold} and \ref{sec:threefold}. Let us remark that Jackson's \cite{J} well-known $q$-analogs of Bessel functions are included in this scheme: the Stieltjes-Wigert case is closely related to Jackson's second $q$-Bessel function, and the Hahn-Exton $q$-Bessel function is Jackson's third $q$-Bessel function. Furthermore, Jackson's first and second $q$-Bessel functions are basically the same function, see e.g.~\cite[Theorem 14.1.3]{Ism}.

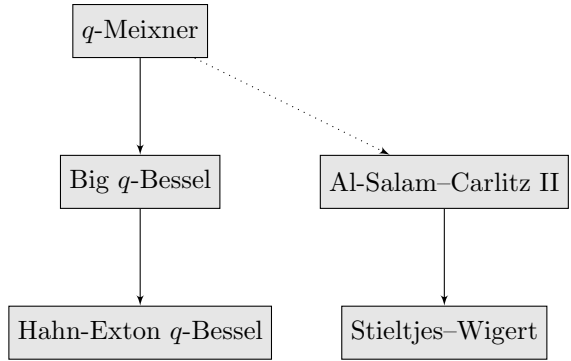
\begin{figure}[h] \label{fig:qBessel}
\tikzstyle{int}=[draw, fill=gray!20, minimum size=2em]
\begin{tikzpicture}[node distance=2cm,auto,>=latex']
    \node [int] (a) {$q$-Meixner};
    \node [int] (b) [below of=a] {Big $q$-Bessel};
    \node [int] (c) [below of=b] {Hahn-Exton $q$-Bessel};
    \node [int] (d) [right of=b, node distance=4cm]{Al-Salam--Carlitz II};
    \node [int] (e) [below of=d] {Stieltjes--Wigert};
    \path[->] (a) edge node{} (b);
    \path[->] (b) edge node{} (c);
    \path[->, dotted] (a) edge node{} (d);
    \path[->] (d) edge node{} (e);
\end{tikzpicture}
\caption{$q$-Analogs of Bessel functions}
\end{figure}
The solid arrows in the scheme correspond to limit relations between elements in the algebra $\mathcal A_q$. The left side of the scheme contains coupling coefficients for three-fold tensor product representation, the right side for two-fold tensor products.

The scheme in Figure \ref{fig:qBessel} can be considered as being part of an extended Askey-scheme of $q$-hypergeometric integral transforms \cite{KS01}, which contains the well-known Askey-scheme of $q$-hypergeometric polynomials, see \cite{KLS}. Figure \ref{fig:qBessel} is also closely related to the $q$-Meixner scheme of indeterminate moment problems within the Askey-scheme of $q$-hypergeometric orthogonal polynomials, see \cite[p.24]{Chr}. Indeed, the $q$-Meixner functions are $q$-Meixner polynomials for specific values of the spectral variable, and therefore the orthogonality measure for the $q$-Meixner functions gives a solution for the corresponding indeterminate moment problem, see \cite{GroenK}. The other cases in Figure \ref{fig:qBessel} also correspond to indeterminate moment problems, except the Hahn-Exton $q$-Bessel functions.

The outline of the paper is as follows. In Section \ref{sec:quantumSU(2)} we give preliminaries on representations of $\mathcal A_q(\mathrm{SU}(2))$. In particular we consider the tensor product of two infinite dimensional irreducible $*$-representations and corresponding eigenvectors, and we introduce useful subalgebras depending on two parameters $\tau,\si \in \R$. In Sections \ref{sec:tauinfty-CGC} and \ref{sec:tausi-CGC} we determine explicitly the Clebsch-Gordan coefficients corresponding to eigenvectors of $\rho_{\tau,\si}$. First in Section \ref{sec:tauinfty-CGC} we do this in case $\si=\infty$, and in Section \ref{sec:tausi-CGC} the general case is considered. In Section \ref{sec:twofold} we first determine eigenvectors of $\rho_{\tau,\si}$, in case of a two-fold tensor product representation, which are different from the eigenvectors from Section \ref{sec:tauinfty-CGC}. Then we determine explicitly coupling coefficients between the different eigenvectors. This leads to a $q$-integral identity involving several types of $q$-hypergeometric orthogonal polynomials and functions closely related to the Al-Salam--Carlitz II polynomials. In section \ref{sec:threefold} we first determine eigenvectors for $\rho_{\tau,\si}$ in a three-fold tensor product representation. We show that the coupling coefficients between different eigenvectors are given in terms of $q$-Meixner functions. This leads to more $q$-integral identities.

\*\\
\textbf{Notation.}
We denote the set of integers by $\Z$, and by $\N$ we denote the set of nonnegative integers. $\T=\{x \in \C \mid |x|=1\}$, the unit circle in $\C$. By $\sqrt{\cdot}$ we denote the principal branch of the square root. Throughout the paper $q$ is a fixed number in $(0,1)$. We use notations for $q$-shifted factorials, $\te$-functions and $q$-hypergeometric functions as in \cite{GR}, i.e.,
\begin{gather*}
(a;q)_\infty = \prod_{j=0}^{\infty} (1-aq^j),\quad (a;q)_n = \frac{ (a;q)_\infty}{(aq^n;q)_\infty} \qquad n \in \Z,\\
\te(a;q) = (a;q)_\infty (q/a;q)_\infty, \qquad a \not\in q^\Z,\\
(a_1,a_2,\ldots, a_k;q)_n = \prod_{j=1}^k (a_j;q)_n, \quad
\te(a_1,a_2,\ldots,a_k;q)= \prod_{j=1}^k \te(a_j;q),\\
\rphis{r}{s}{a_1,a_2,\ldots,a_r}{b_1,b_2,\ldots,b_s}{q,z} = \sum_{n=0}^\infty \frac{(a_1,a_2,\ldots,a_r;q)_n }{ (q,b_1,b_2,\ldots,b_s;q)_n}\left( (-1)^n q^{\frac12n(n-1)} \right)^{s-r-1} z^n,
\end{gather*}
and we also use the shorthand notations
\[
(ab^{\pm 1};q)_n = (ab,a/b;q)_n, \qquad \te(ab^{\pm 1};q) = \te(ab,a/b;q).
\]
An identity that we frequently use, often without mentioning, is the $\te$-product identity:
\begin{equation} \label{eq:te-product}
\te(xq^k;q) = (-x)^{-k} q^{-\frac12 k(k-1)} \te(x), \qquad k \in \Z.
\end{equation}
We also need the Jackson $q$-integral, see \cite{GR}, which is defined by
\begin{gather*}
\int_{0}^\al f(x) \, d_qx =  (1-q) \sum_{k=0}^\infty f(\al q^k)\al q^k, \\
\int_\be^\al f(x) \, d_qx =  \int_0^\al f(x) \, d_qx - \int_0^\be f(x) \, d_qx,\\
\int_0^{\infty(\al)} f(x) \, d_qx = (1-q) \sum_{k=-\infty}^\infty f(\al q^k) \al q^k,\\
\int_\be^{\infty(\al)} f(x) \, d_qx = \int_0^{\infty(\al)} f(x) \, d_qx\, -\,
\int_{0}^\be f(x) \, d_qx, \\
\int_{\infty(\be)}^{\infty(\al)} f(x) \, d_qx = \int_0^{\infty(\al)} f(x) \, d_qx\, - \,
\int_0^{\infty(\be)} f(x) \, d_qx,
\end{gather*}
for a function $f$ such that the sums converge absolutely. Observe that for $\al>0$ and $\be<0$,
\[
\int_\be^\al f(x) d_qx = (1-q)\sum_{x \in \be q^\N \cup \al q^\N} f(x) |x|.
\]

\section{The ${\mathrm{SU}(2)}$ quantum group} \label{sec:quantumSU(2)}
The ${\mathrm{SU}(2)}$ quantum group is the complex unital associative algebra $\mathcal A_q=\mathcal A_q(\mathrm{SU}(2))$  generated by $\al$, $\be$, $\ga$, $\de$, which satisfy the relations
\begin{equation} \label{eq:comm rel}
\begin{split}
\al \be = q \be \al, \quad \al \ga = q \ga \al, \quad \be \de = q \de \be, \quad \ga \de = q \de \ga, \\
\be \ga = \ga \be, \quad \al \de - q\be \ga = 1 = \de \al - q^{-1}\be \ga.
\end{split}
\end{equation}
There is a $*$-structure defined on the generators by
\begin{equation} \label{eq:*structure}
\al^* = \de, \quad \be^* = -q \ga, \quad \ga^* = -q^{-1} \be, \quad \de^* = \al.
\end{equation}
$\mathcal A_q$ is a Hopf-$*$-algebra with comultiplication $\De$, which is defined on the generators by
\begin{equation} \label{eq:coprod}
\begin{split}
\De(\al) = \al \tensor \al + \be \tensor \ga, \quad \De(\be) = \al \tensor \be + \be \tensor \de,\\
\De(\ga) = \ga \tensor \al + \de \tensor \ga, \quad \De(\de) = \de \tensor \de + \ga \tensor \be.
\end{split}
\end{equation}
We do not need the antipode in this paper. The irreducible $*$-representation of $\mathcal A_q$ are either 1-dimensional, or infinite dimensional. The infinite dimensional irreducible $*$-representations are labeled by $\phi \in [0,2\pi)$, and we denote a representation by $\pi_\phi$. The representation space of $\pi_\phi$ is $\ell^2(\N)$, and the generators $\al,\be,\ga,\de$ act on the standard orthonormal basis $\{e_n\}_{n \in \N}$ of $\ell^2(\N)$ as raising and lowering operators:
\begin{equation} \label{eq:repr}
\begin{split}
\pi_\phi(\al)\, e_n &= \sqrt{1-q^{2n}}\, e_{n-1},\\
\pi_\phi(\be)\,e_n &= - e^{-i\phi} q^{n+1}\, e_n,\\
\pi_\phi(\ga)\, e_n &= e^{i\phi} q^n\, e_n,\\
\pi_\phi(\de)\, e_n &= \sqrt{1-q^{2n+2}} e_{n+1}.
\end{split}
\end{equation}
Note that the actions of $\be$ and $\de$ can be found from the actions of $\al$ and $\ga$ using the $*$-structure.

\subsection{The tensor product representation} \label{sec:representationT}
In this paper we are mainly interested in the tensor product representation
\begin{equation} \label{def:representation T}
\mathcal T = \mathcal T_{\phi,\psi} = (\pi_\phi \tensor \pi_\psi)  \De, \qquad \phi,\psi \in [0,2\pi),
\end{equation}
of $\mathcal A_q$ on $\ell^2(\N) \tensor \ell^2(\N)$. Koornwinder \cite{Koo91} determined the decomposition of $\mathcal T$ into irreducible representations. The result is as follows.

Let $\rho$ be the direct integral representation $\rho =\int_0^{2\pi} \pi_\phi\, d\phi$ acting on $L^2(0,2\pi) \tensor \ell^2(\N) \cong  \int_0^{2\pi} \ell^2(\N) d\phi$. Let $\{v_m\}_{m \in \Z}$ denote the standard orthonormal basis of $L^2(0,2\pi)$, $v_m(\phi)= \frac{1}{\sqrt{2\pi}}e^{-im\phi}$. One checks directly the actions of the generators on standard basis elements:
\begin{equation} \label{eq:dirint}
\begin{split}
\rho(\al)\, v_m \tensor e_n &= \sqrt{ 1- q^{2n}}\, v_m \tensor e_{n-1},\\
\rho(\be)\, v_m \tensor e_n &= -q^{n+1}\, v_{m+1} \tensor e_{n},\\
\rho(\ga)\, v_m \tensor e_n &= q^n\,v_{m-1} \tensor e_n,\\
\rho(\de)\, v_m \tensor e_n &= \sqrt{ 1- q^{2n+2}}\, v_m \tensor e_{n+1}.
\end{split}
\end{equation}
\begin{thm} \label{thm:tensor product decomp}
For $\phi,\psi \in [0,2\pi)$ the tensor product representation $\mathcal T_{\phi,\psi}$ is unitarily equivalent to the direct integral representation $\rho$.
\end{thm}

It is instructive to go through the proof of Theorem \ref{thm:tensor product decomp}. We compute the Clebsch-Gordan coefficients corresponding to the standard basis $\{e_n\}_{n \in \N}$ of $\ell^2(\N)$. Observe that by \eqref{eq:repr} the standard basis vector $e_n$ is an eigenvector of $\pi_\phi(\ga \ga^*)$ for eigenvalue $q^{2n}$. We consider $\mathcal T(\ga \ga^*)$ acting on basis elements $e_{n_1} \tensor e_{n_2}$. From \eqref{eq:coprod} we obtain
\[
\De(\ga \ga^*) = -q^{-1}\Big( \ga \be \tensor \al \de+ \ga \al \tensor \al \be + \de \be \tensor \ga \de + \de \al \tensor \ga \be\Big),
\]
and then \eqref{eq:repr} gives
\[
\begin{split}
\mathcal T(\ga\ga^*) \,e_{n_1} \tensor e_{n_2}=&\, e^{i(\phi-\psi)} q^{n_1+n_2-1} \sqrt{ (1-q^{2n_1})(1-q^{2n_2})} \, e_{n_1-1} \tensor e_{n_2-1} \\
&+ [ q^{2n_1}+ q^{2n_2} - q^{2n_1+2n_2}(1+q^2)]\, e_{n_1} \tensor e_{n_2}\\
& + e^{-i(\phi-\psi)} q^{n_1+n_2+1} \sqrt{ (1-q^{2n_1+2})(1-q^{2n_2+2})} \, e_{n_1+1} \tensor e_{n_2+1}.
\end{split}
\]
We define
\[
f_n^p =
\begin{cases}
e_n \tensor e_{n+p}, & p \geq 0,\\
e_{n-p} \tensor e_n, & p <0,
\end{cases}
\]
and let $H_p$, $p \in \Z$, be the Hilbert space defined by
\[
H_p = \overline{ \text{span}}\{ f_n^p\ | \ n \in \N \} \cong \ell^2(\N).
\]
Then we see that $\mathcal T(\ga\ga^*)$ leaves $H_p$ invariant, and restricted to $H_p$ it acts as a Jacobi operator (or tridiagonal operator). To diagonalize $\mathcal T(\ga\ga^*)|_{H_p}$ we need the Wall polynomials, see \cite{KLS}, which are defined by
\begin{equation} \label{eq:defWallpol}
p_n(y;a;q) = \rphis{2}{1}{ q^{-n}, 0 }{aq}{q, qy}.
\end{equation}
For $0<a<q^{-1}$ these polynomials satisfy the orthogonality relations,
\[
\sum_{x\in \N} p_m(q^x;a;q) p_n(q^x;a;q) \frac{ (aq)^x} {(q;q)_x}= \de_{mn} \frac{ (aq)^n (q;q)_n}{(aq;q)_\infty (aq;q)_n},
\]
and they form a basis for the corresponding weighted $\ell^2$-space. Let the function $\bar p_n(q^x;a;q)$ be defined by
\begin{equation} \label{eq:defbarpn}
\bar p_n(q^x;a;q)=(-1)^{n+x} \sqrt{ \frac{ (aq)^{x-n} (aq;q)_\infty (aq;q)_n }{ (q;q)_n (q;q)_x } }\, p_n(q^x;a;q),
\end{equation}
then from the orthogonality relation for the Wall polynomials and from completeness we obtain the orthogonality relations
\[
\sum_{x\in \N} \bar p_n(q^x;a;q) \bar p_m(q^x;a;q) = \de_{nm}, \qquad \sum_{n \in \N} \bar p_n(q^x;a;q) \bar p_n(q^y;a;q) = \de_{xy},
\]
for $0<a<q^{-1}$. The latter, the dual orthogonality relations, actually correspond to orthogonality relations for the Al-Salam--Carlitz II polynomials, see \cite{KLS}.
The three-term recurrence relation for the Wall polynomials is equivalent to
\[
\begin{split}
q^x \bar p_n(q^x;a;q) =& q^{n+\hf}\sqrt{ a(1-q^n)(1-aq^{n+1}) } \bar p_{n+1}(q^x;a;q) \\
&+ \big[q^n(1-aq^{n+1})+aq^n(1-q^n)\big] \bar p_n(q^x;a;q)\\
&+ q^{n-\hf}\sqrt{ a(1-q^n)(1-aq^n)} \bar p_{n-1}(q^x;a;q),
\end{split}
\]
with $\bar p_{-1}(q^x)=0$ and $\bar p_0(q^x)=(-1)^x \sqrt{(aq)^x (aq;q)_\infty / (q;q)_x }$. We define
\begin{equation} \label{eq:Clebsch-Gordan inftyinfty}
c_{x,p,n} =
\begin{cases}
\bar p_n(q^{2x};q^{2p};q^2), & p\in \N,\\
\bar p_n(q^{2x};q^{-2p};q^2) , & -p\in \N,
\end{cases}
\end{equation}
then it follows that the vector
\begin{equation} \label{eq:eigenvector V infty infty}
V_{x,p} =
\begin{cases}
\displaystyle \sum_{n \in \N} e^{in\psi-i(n+p)\phi+ix(\phi-\psi)} c_{x,p,n}\, e_{n}\tensor e_{n+p}, & p \geq 0,\\ \\
\displaystyle \sum_{n \in \N} e^{i(n-p)\psi-in\phi+ix(\phi-\psi)} c_{x,p,n}\, e_{n-p} \tensor e_{n}, & p \leq 0,
\end{cases}
\end{equation}
is an eigenvector of $\mathcal T(\ga\ga^*)$ for eigenvalue $q^{2x}$, $x \in \N$. We define $e_n = 0$ for $n \in -\N_{\geq 1}$, then the eigenvector $V_{x,p}$ can actually be defined for all $p\in \Z$ by the first expression in \eqref{eq:eigenvector V infty infty}. This is a consequence of the identity
\[
\bar p_n(q^x;q^{-p};q) =
\begin{cases}
\bar p_{n-p}(q^x;q^{p};q), & 0 \leq p \leq n,\\
0, & p>n,
\end{cases}
\]
which follows from the $_2\varphi_1$-expression for $\bar p_n$, and, for $p \in \N$,
\begin{equation} \label{eq:sum p<->-p}
(q^{1-p};q)_\infty \sum_{n=0}^\infty \frac{ A_n }{(q,q^{1-p};q)_n } = \sum_{n=p}^{\infty} \frac{ A_n (q^{1-p+n};q)_\infty }{(q;q)_n} = (q^{1+p};q)_\infty \sum_{n=0}^\infty \frac{ A_{n+p} }{(q,q^{1+p};q)_n }.
\end{equation}
From the dual orthogonality relations for $\bar p_n$ it follows that $\{V_{x,p}\mid x \in \N, p \in \Z\}$ is an orthonormal basis for $\ell^2(\N) \tensor \ell^2(\N)$, and we see that \eqref{eq:eigenvector V infty infty} is equivalent to
\[
e_{n_1} \tensor e_{n_2} = \sum_{x \in \N}e^{in_2\phi -in_1\psi-ix(\phi-\psi)} c_{x,|n_2-n_1|,n_1} V_{x,n_1}.
\]

Finally, we need to determine the actions of the $\mathcal A_q$-generators on $V_{x,p}$. Note that $\mathcal T(\de)$ leaves $H_p$ invariant, and from the commutation relations \eqref{eq:comm rel} we see that $\de (\ga \ga^*) = q^{-2} (\ga \ga^*) \de$, so $\mathcal T(\de) V_{x,p}$ is an eigenvector in $H_p$ of $\mathcal T(\ga \ga^*)$ for eigenvalue $q^{2x+2}$, which implies $\mathcal T(\de) V_{x,p} = C V_{x+1,p}$ for some $C \in \C$. To determine the value of $C$ we use $\de^*=\al$ and $\De(\al) = \al\tensor \al + \be \tensor \ga$, see \eqref{eq:*structure} and \eqref{eq:coprod},
\[
\begin{split}
C\langle V_{x+1,p}, e_0 \tensor e_p \rangle & = \langle \mathcal T(\de) V_{x,p}, e_0 \tensor e_p \rangle  = \langle V_{x,p}, \mathcal T(\al) e_0 \tensor e_p \rangle \\
&= - e^{i(\phi-\psi)}q^{p+1}\langle V_{x,p}, e_0 \tensor e_{p} \rangle.
\end{split}
\]
From $c_{x,p,0} = (-1)^x q^{x(p+1)} \sqrt{ (q^{2p+2};q)_\infty / (q^2;q^2)_x}$ it now follows that $C = \sqrt{ 1-q^{2x+2}}$. Using $\al\de = 1+ q^2 \ga \ga^*$, the action of $\al$ is obtained from the action of $\de$. For the action of $\ga$ we note that $\mathcal T(\ga)$ sends $H_p$ to $H_{p-1}$, and from $\ga(\ga \ga^*) = (\ga \ga^*) \ga$ it then follows that $\mathcal T(\ga) V_{x,p}$ is an eigenvector in $H_{p-1}$ of $\mathcal T(\ga \ga^*)$ for eigenvalue $q^{2x}$, so $\mathcal T(\ga) V_{x,p}= D V_{x,p-1}$ for some $D \in \C$. The value of $D$ can be computed similarly as above. De the action of $\be$ can be obtained from the action of $\ga$. We now obtained
\begin{equation} \label{eq:action on V infty infty}
\begin{split}
\mathcal T(\al) \, V_{x,p} &= \sqrt{1-q^{2x}} \, V_{x-1,p}, \\
\mathcal T(\be) \, V_{x,p} &= q^{x+1}\, V_{x,p+1},\\
\mathcal T(\ga) \, V_{x,p} &= q^x\, V_{x,p-1},\\
\mathcal T(\de) \, V_{x,p} &= \sqrt{1-q^{2x+2}} \, V_{x+1,p}.
\end{split}
\end{equation}
Comparing this with the actions of the generators $\al$, $\be$, $\ga$, $\de$ in the direct integral representation $\rho$, we see that the unitary operator $\La:\ell^2(\N) \tensor \ell^2(\N) \to L^2(0,2\pi) \tensor \ell^2(\N)$ defined by
\[
\La V_{x,p} = v_p \tensor e_x,
\]
intertwines $\mathcal T(X)$ with $\rho(X)$ for any $X \in \mathcal A_q$. This proves Theorem \ref{thm:tensor product decomp}.

\begin{remark}
The actions of $\al,\be,\ga,\de$ on the eigenvectors $V_{x,p}$ imply contiguous relations for the Clebsch-Gordan coefficients $c_{x,p,n}$, which in turn imply the following contiguous relations for Wall polynomials:
\[
\begin{split}
(1-a)p_n(q^x;a/q;q)&=(1-aq^n)p_n(q^x;a;q) -  a(1-q^n) p_{n-1}(q^x;a;q),\\
q^x p_n(q^x;aq;q) &= q^n(1-aq) \big[p_{n}(q^x;a;q) - p_{n+1}(q^x;a;q)\big],\\
(1-q^x) p_n(q^{x-1};a;q) &= (1-aq^{n+1})p_{n+1}(q^x;a;q) + aq^{n+1} p_n(q^x;a;q).
\end{split}
\]
These relations can of course also be proved directly. To obtain the the first relation, we expand $p_n(q^x;a/q;q)$ in terms of $p_n(q^x;a;q)$,
\[
p_n(q^x;a/q;q) = \sum_{k=0}^n c_k p_k(q^x;a;q).
\]
The coefficient $c_n$ can be found by comparing leading coefficients. Next we multiply both sides with $p_{m}(q^x;a;q)w(q^x;a;q)$, $0 \leq m \leq n-1$, and sum over  $x$ from $0$ to $\infty$. Using the orthogonality relation for the Wall polynomials on both sides, we find the value of $c_{n-1}$, and we find $c_k=0$ for $k \leq n-2$. The other contiguous relations are proved in the same way. Note that, having the contiguous relations, the actions of the generators on $V_{x,p}$ can be derived from them.
\end{remark}

\subsection{Special elements in $\mathcal A_q$}
For $\tau,\si \in \R$ the element $\rho_{\tau,\si} \in \mathcal A_q$ is defined by
\begin{equation} \label{def:rhotausigma}
\begin{split}
\rho_{\tau,\si}  = \frac12\Big(& \al^2 + \de^2 + q\ga^2 + q^{-1}\be^2 + i(q^{-\si}-q^\si) (q\de\ga+\be \al) \\ &-  i(q^{-\tau}-q^\tau) (\de\be+ q \ga\al) + (q^{-\si}-q^\si)(q^{-\tau}-q^\tau)\Big),
\end{split}
\end{equation}
and furthermore,
\begin{equation} \label{def:rhoinftysigma}
\begin{split}
\rho_{\tau,\infty} &= \lim_{\si \rightarrow \infty} 2q^{\tau+\si-1} \rho_{\tau,\si} =  iq^{\tau}(\de \ga + q^{-1}\be \al) + q^{-1}(1-q^{2\tau})\ga \be, \\
\rho_{\infty,\si} &= \lim_{\tau \rightarrow \infty} 2q^{\tau+\si-1} \rho_{\tau,\si} =  -iq^{\si}( q^{-1}\de \be + \ga \al) + q^{-1}(1-q^{2\si})\ga \be.
\end{split}
\end{equation}
Note that $\rho_{\tau,\si}^* = \rho_{\tau,\si}$. Observe also that $\rho_{\infty,\infty} = \lim_{\tau \to \infty} \rho_{\tau,\infty}=-\ga \ga^*$. The element $\rho_{\tau,\si}$, first introduced by Koornwinder, plays an important role in the harmonic analysis on $\mathcal A_q$, see e.g.~\cite{Koo93},\cite{Koe96}. In this paper we are mainly interested in the spectral analysis of $\mathcal T(\rho_{\tau,\si})$. The following elements of $\mathcal A_q$ will be useful. For $\tau,\si \in \R \cup \{\infty\}$ we define, cf.~\cite[Prop.~6.5]{Koe96},
\begin{equation} \label{eq:defalbegadetausi}
\begin{split}
\al_{\tau,\si} &= q^\hf \al -iq^{\si-\hf}\be +iq^{\tau+\hf} \ga + q^{\si+\tau-\hf} \de,\\
\be_{\tau,\si} &= -q^{\si+\hf} \al -iq^{-\hf} \be - iq^{\si+ \tau+\hf} \ga + q^{\tau-\hf} \de, \\
\ga_{\tau,\si} &= -q^{\tau+\hf} \al+iq^{\tau+\si-\hf}\be +iq^\hf \ga + q^{\si-\hf} \de, \\
\de_{\tau,\si} &= q^{\tau+\si+\hf} \al +iq^{\tau-\hf} \be -iq^{\si+\hf} \ga + q^{-\hf} \de,
\end{split}
\end{equation}
where $q^\infty =0$. We collect a few useful relations.\\
Adjoints:
\begin{equation} \label{eq:adjointstausi}
\al_{\tau,\si}^* = q \de_{\tau-1,\si-1}, \quad \be_{\tau,\si}^* =- \ga_{\tau-1,\si+1}, \quad \ga_{\tau,\si}^* = - \be_{\tau+1,\si-1}, \quad \de_{\tau,\si}^*= q^{-1} \al_{\tau+1,\si+1}.
\end{equation}
Commutation relations:
\begin{equation} \label{eq:commrelation1}
\begin{split}
\be_{\tau+1,\si-1} \ga_{\tau,\si} &= 2q^{\tau+\si}\rho_{\tau,\si} - q^{2\si-1}-q^{2\tau+1},\\
\ga_{\tau-1,\si+1} \be_{\tau,\si} &= 2q^{\tau+\si}\rho_{\tau,\si} - q^{2\si+1}-q^{2\tau-1},\\
\al_{\tau+1,\si+1} \de_{\tau,\si} & = 2q^{\tau+\si+1}\rho_{\tau,\si} + 1+q^{2\tau+2\si+2},\\
\de_{\tau-1, \si-1} \al_{\tau,\si} & = 2q^{\tau+\si-1}\rho_{\tau,\si} + 1+q^{2\tau+2\si-2},
\end{split}
\end{equation}
\begin{equation} \label{eq:commrelation2}
\begin{aligned}
\al_{\tau,\si} \rho_{\tau,\si} &= \rho_{\tau-1, \si-1}\al_{\tau,\si},& \be_{\tau,\si} \rho_{\tau,\si} &= \rho_{\tau-1,\si+1} \be_{\tau,\si}, \\
\ga_{\tau,\si} \rho_{\tau,\si} &= \rho_{\tau+1, \si-1} \ga_{\tau,\si}, & \de_{\tau,\si} \rho_{\tau,\si} &= \rho_{\tau+1, \si+1} \de_{\tau,\si},
\end{aligned}
\end{equation}
Decompositions:
\begin{equation} \label{eq:decomptau}
\begin{aligned}
\al_{\tau,\si}&= \al_{\tau,\infty} + q^\si \be_{\tau,\infty}, &
\be_{\tau,\si}&= \be_{\tau,\infty} - q^{\si} \al_{\tau,\infty},\\
\ga_{\tau,\si}&= \ga_{\tau,\infty} + q^\si \de_{\tau,\infty}, &
\de_{\tau,\si}&= \de_{\tau,\infty} - q^\si \ga_{\tau, \infty},
\end{aligned}
\end{equation}
\begin{equation} \label{eq:decompsi}
\begin{aligned}
\al_{\tau,\si}&= \al_{\infty,\si} + q^\tau \ga_{\infty,\si}, &
\be_{\tau,\si}&= \be_{\infty,\si} + q^{\tau} \de_{\infty,\si},\\
\ga_{\tau,\si}&= \ga_{\infty,\si} - q^\tau \al_{\infty,\si}, &
\de_{\tau,\si}&= \de_{\infty,\si} - q^\tau \be_{\infty,\si}.
\end{aligned}
\end{equation}
Now $\rho_{\tau,\si}$ can be expressed as
\begin{equation} \label{eq:rho si tau 2}
2\rho_{\tau,\si} = q^{-\tau-\si-1}(\al_{\tau+1,\infty}- q^{\si+1} \be_{\tau+1,\infty} )( \de_{\tau,\infty} - q^{\si} \ga_{\tau,\infty}) + q^{-\tau-\si-1} + q^{\tau+\si+1}.
\end{equation}
Coproduct: for $\la \in \R \cup \{\infty\}$
\begin{equation} \label{eq:coproducttausi}
\begin{split}
\De(\al_{\tau,\si}) &=  \frac{1}{1+q^{2\la+1}}  \left( q^{-\hf} \al_{\tau, \la+1} \tensor \al_{\la+1,\si} + q^\hf \be_{\tau, \la} \tensor \ga_{\la, \si} \right),\\
\De(\be_{\tau,\si}) &= \frac{1}{1+q^{2\la+1}}  \left( q^{-\hf} \al_{\tau,\la+1} \tensor \be_{\la+1,\si} + q^\hf \be_{\tau,\la} \tensor \de_{\la,\si} \right),\\
\De(\ga_{\tau,\si}) &= \frac{1}{1+q^{2\la+1}}  \left( q^{-\hf} \ga_{\tau,\la+1} \tensor \al_{\la+1,\si} + q^\hf \de_{\tau,\la} \tensor \ga_{\la,\si} \right),\\
\De(\de_{\tau,\si}) &= \frac{1}{1+q^{2\la+1}}  \left( q^{-\hf} \ga_{\tau,\la+1} \tensor \be_{\la+1,\si}+ q^\hf \de_{\tau,\la} \tensor \de_{\la ,\si}  \right).
\end{split}
\end{equation}
This gives us
\begin{equation} \label{eq:coproduct rho tau sigma}
\begin{split}
2\De(\rho_{\tau,\si}+&q^{-\tau-\si-1}+q^{\tau+\si+1})  = \\
\frac{ q^{-\tau - \si -1} }{(1+q^{2\la})^2} &\Big( q^{-1} \al_{\tau+1,\la+1}\ga_{\tau,\la+1}\tensor \al_{\la+1,\si+1}\be_{\la+1,\si}
 + q\, \be_{\tau+1,\la}\de_{\tau,\la}\tensor \ga_{\la,\si+1}\de_{\la,\si}\\
&+[2q^{\tau+\la+1}\rho_{\tau,\la}+q^{2\la+2\tau+2}+1] \tensor [2q^{\si+\la+1}\rho_{\la,\si}+q^{2\la+2\si+2}+1]\\
&+[2q^{\la+\tau+1} \rho_{\tau,\la+1}-q^{2\la+1}-q^{2\tau+1}] \tensor [2q^{\la+\si+1} \rho_{\la+1,\si}-q^{2\la+1}-q^{2\si+1}]\Big).
\end{split}
\end{equation}
Let us denote by $\mathcal A_q^{\tau,\si}$ the subalgebra of $\mathcal A_q$ generated by
\[
\al_{\tau+n,\si+m},\ \be_{\tau+n,\si+m},\ \ga_{\tau+n,\si+m},\ \de_{\tau+n,\si+m}, \qquad n,m \in \Z,
\]
and subalgebras $\mathcal A_q^{\tau,\infty}$, $\mathcal A_q^{\infty,\si}$ are defined similarly. Observe that $\mathcal A_q \supset \mathcal A_q^{\tau,\infty} \supset \mathcal A_q^{\tau,\si}$ and $\mathcal A_q \supset \mathcal A_q^{\infty,\si} \supset \mathcal A_q^{\tau,\si}$. This structure allows us to `build up' representations of $\mathcal A_q^{\tau,\si}$ from (simpler) representations of $\mathcal A_q$, $\mathcal A_q^{\tau,\infty}$ and $\mathcal A_q^{\infty,\si}$. It has been shown by Stokman in \cite{Sto03} that the algebra $\mathcal A_q^{\tau,\si}$ is closely related to the $SU(2)$ dynamical quantum group \cite {EtVar}, \cite{KoeRos}.

\section{$(\tau,\infty)$-Clebsch-Gordan coefficients} \label{sec:tauinfty-CGC}
In this section we compute Clebsch-Gordan coefficients corresponding to eigenvectors of $\mathcal T(\rho_{\tau,\infty})$.

\subsection{Eigenvectors of $\mathbf{\rho_{\tau,\infty}}$}
Eigenvectors of $\pi_\phi(\rho_{\tau,\infty})$ can be found in \cite[\S5]{KoeVer}. We use a different normalization here. To define the eigenvectors we need the Al-Salam--Carlitz polynomials, see \cite{ASC},\cite{KLS}, defined by
\begin{equation} \label{eq:Al-Salam--Carlitz pol}
U_n^{(a)}(x;q) = (-a)^n q^{\hf n(n-1)} \rphis{2}{1}{q^{-n}, x^{-1}}{0}{q, \frac{qx}{a}}.
\end{equation}
They satisfy the three-term recurrence relation
\[
x U_n^{(a)}(x;q) = U_{n+1}^\al(x;q) +(a+1)q^n U_n^{(a)}(x;q) -aq^{n-1}(1-q^n)U_{n-1}^{(a)}(x;q),
\]
and the orthogonality relation
\[
\begin{split}
\int_a^1 (qx,qx/a;q)_\infty U_n^{(a)}(x;q) U_m^{(a)}(x;q) \,d_qx =\de_{mn}\,(-a)^n (1-q) (q;q)_n (q;q)_\infty \te(a;q) q^{\hf n(n-1)}
\end{split}
\]
for $a<0$.
Since the polynomials $U_n^{(a)}$ are orthogonal on a countable set and complete, they also satisfy the dual orthogonality relations
\[
\sum_{n \in \N}  U_n^{(a)}(x;q) U_n^{(a)}(y;q)\frac{ q^{-\hf n(n-1)}}{ (-a)^n (q;q)_n } = \de_{xy} \frac{ (q;q)_\infty \te(a;q)}{|x|(qx,qx/a;q)_\infty},
\]
provided $a<0$ and $x,y \in aq^{\N} \cup q^\N$.

From the definition \eqref{def:rhoinftysigma} of $\rho_{\tau,\infty}$ and the actions \eqref{eq:repr} of the $\mathcal A_q$-generators we obtain
\begin{equation} \label{eq:actionrhotauinfty}
\pi_\phi(\rho_{\tau,\infty}) e_n = iq^\tau e^{i\phi} q^n \sqrt{1-q^{2n+2}}\, e_{n+1} - q^{2n}(1-q^{2\tau}) \, e_n -iq^\tau e^{-i\phi} q^{n-1} \sqrt{1-q^{2n}}\, e_{n-1}.
\end{equation}
We can diagonalize this using the Al-Salam--Carlitz polynomials. We define the function $m_{x,n}^{\tau,\infty}$ by
\[
m_{x,n}^{\tau,\infty} = i^n \bar{U}_n^{(-q^{2\tau})}(-x;q^2),
\]
where
\[
\bar U^{(a)}_n(x;q) = \sqrt{ \frac{q^{-\hf n(n-1)}|x|(qx,qx/a;q)_\infty}{ (-a)^n (q;q)_n(q;q)_\infty \te(a;q)} } \, U^{(a)}_n(x;q).
\]
The three-term recurrence relation for the Al-Salam--Carlitz polynomials translates to
\begin{equation} \label{eq:3term m tau infty}
\begin{split}
-x m_{x,n}^{\tau,\infty} = -iq^{n+\tau}\sqrt{1-q^{2n+2} }m_{x,n+1}^{\tau,\infty}+(1-q^{2\tau})q^{2n} m_{x,n}^{\tau,\infty}(x) + i q^{n-1+\tau}\sqrt{1-q^{2n} }m_{x,n-1}^{\tau,\infty}.
\end{split}
\end{equation}
Comparing this with \eqref{eq:actionrhotauinfty} we find the following result.
\begin{prop} \label{prop:eigenvector rho tau infty}
For $x \in -q^{2\N} \cup q^{2\tau+2\N}$ the vector
\[
v_{x}^{\tau,\infty} = \sum_{n \in \N} e^{in\phi} m_{x,n}^{\tau,\infty}\, e_n
\]
is an eigenvector of $\pi_\phi(\rho_{\tau,\infty})$ for eigenvalue $x$. Moreover, $\{v_{x}^{\tau,\infty} \ |\ x \in -q^{2\N} \cup q^{2\tau+2\N} \}$ is an orthonormal basis for $\ell^2(\N)$.
\end{prop}
\begin{proof}
It follows from the action of $\rho_{\tau,\infty}$ on the standard orthonormal basis and from \eqref{eq:3term m tau infty} that $v_x^{\tau,\infty}$ is an eigenvector. Using the dual orthogonality relations for the Al-Salam--Carlitz polynomials we obtain orthogonality relations for the eigenvectors.
\end{proof}

We can determine the explicit actions of $\al_{\tau,\infty}, \be_{\tau,\infty}, \ga_{\tau,\infty}, \de_{\tau,\infty}$ on the eigenvectors $v_x^{\tau,\infty}$.
\begin{prop} \label{prop:actions on v tau infty}
For $x \in  -q^{2\N} \cup q^{2\tau+2\N}$,
\[
\begin{split}
\pi_\phi(\al_{\tau,\infty})\, v_{x}^{\tau,\infty} &= i e^{i\phi}q^\hf \sqrt{ 1+x} \, v_{x/q^2}^{\tau-1,\infty},\\
\pi_\phi(\be_{\tau,\infty})\, v_{x}^{\tau,\infty} &=i e^{-i\phi} q^{\tau-\hf} \sqrt{ 1-x q^{2-2\tau}}\, v_{x}^{\tau-1,\infty},\\
\pi_\phi(\ga_{\tau,\infty})\, v_{x}^{\tau,\infty} &= i e^{i\phi} q^{\tau+\hf} \sqrt{ 1-x q^{-2\tau}}\, v_{x}^{\tau+1,\infty},\\
\pi_\phi(\de_{\tau,\infty})\, v_{x}^{\tau,\infty} &=-i e^{-i\phi}q^{-\hf} \sqrt{ 1+x q^2}\, v_{x q^2}^{\tau+1,\infty},
\end{split}
\]
where $v^{\tau,\infty}_{-q^{-2}} = v^{\tau,\infty}_{q^{2\tau-2}} = 0$.
\end{prop}
\begin{proof}
From $\de_{\tau,\infty} \rho_{\tau,\infty} = q^{-2} \rho_{\tau+1, \infty} \de_{\tau,\infty}$, see \eqref{eq:commrelation2}, we obtain $\pi_\phi(\de_{\tau,\infty}) v_x^{\tau,\infty} = C v_{x q^2}^{\tau+1,\infty}$ for some constant $C$. Using \eqref{eq:adjointstausi} and \eqref{eq:decompsi} with $\si=\infty$, the value of $C$ follows from
\[
\begin{split}
C \langle v_{x q^2}^{\tau+1,\infty}, e_0 \rangle &=  \langle \pi_\phi(\de_{\tau,\infty} ) v_{x}^{\tau,\infty}, e_0 \rangle = \langle v_{x}^{\tau+1,\infty}, \pi_\phi(\de_{\tau,\infty}^*) e_0 \rangle \\
&= \langle v_{x}^{\tau,\infty}, \pi_\phi(q^{-\hf} \al + iq^{\tau+\hf} \ga) e_0 \rangle = -ie^{-i\phi}q^{\tau+\hf}\,\langle v_{x}^{\tau,\infty}, e_0 \rangle.
\end{split}
\]
Using $\al_{\tau+1,\infty} \de_{\tau,\infty}=q^2 \rho_{\tau,\infty} +1$, see \eqref{eq:commrelation1}, we obtain the action of $\al_{\tau,\infty}$ from the action of $\de_{\tau,\infty}$. The actions of $\be_{\tau,\infty}$ and $\ga_{\tau,\infty}$ can be obtained in the same way, using $\be_{\tau,\infty} \rho_{\tau,\infty} = \rho_{\tau-1,\infty} \beta_{\tau,\infty}$.
\end{proof}
\begin{remark}
Because of the shift in the parameter $\tau$, the actions of $\al_{\tau,\infty}, \be_{\tau,\infty}, \ga_{\tau,\infty}, \de_{\tau,\infty}$ are different from the actions of the generators $\al$, $\be$, $\ga$, $\de$ on the standard basis of $\ell^2(\N)$: here we have `dynamic' raising and lowering operators.
\end{remark}

For later references we also need an orthonormal basis on which $\pi_\phi(\rho_{\infty,\si})$ acts diagonally. This basis is $\{v_{x}^{\infty,\si}\ |\ x \in -q^{2\N} \cup q^{2\si + 2 \N} \}$, where
\begin{gather*}
v_{x}^{\infty,\si} = \sum_{n \in \N} e^{-in\phi} m_{x,n}^{\infty,\si} \, e_n,\\
m_{x,n}^{\infty,\si} = i^n q^{-n\si} q^{-\frac12 n(n-1)}\sqrt{ \frac{|x|(-q^2x,xq^{2-2\si};q^2)_\infty }{(q^2;q^2)_n (q^2,-q^{2\si},-q^{2-2\si};q)_\infty}}\ U_n^{(-q^{2\si})}(-x;q^2).
\end{gather*}
The vector $v_{x}^{\infty,\si}$ is an eigenvector of $\pi_\phi(\rho_{\infty, \si})$ for eigenvalue $x$, and
\begin{equation} \label{eq:actions on v infty si}
\begin{split}
\pi_\phi(\al_{\infty,\si})\, v_{x}^{\infty,\si} &= i e^{-i\phi}q^\hf \sqrt{ 1+ x }\, v_{x/q^2 }^{\infty,\si-1},\\
\pi_\phi(\be_{\infty,\si})\, v_{x}^{\infty,\si} &= i e^{-i\phi} q^{\si+\hf} \sqrt{ 1- x q^{-2\si} }\, v_{x}^{\infty,\si+1},\\
\pi_\phi(\ga_{\infty,\si})\, v_{x}^{\infty,\si} &= i e^{i\phi} q^{\si-\hf} \sqrt{ 1- x q^{2-2\si} }\, v_{x}^{\infty,\si-1},\\
\pi_\phi(\de_{\infty,\si})\, v_{x}^{\infty,\si} &= -i e^{i\phi} q^{-\hf}\sqrt{ 1+ x q^2 }\, v_{xq^2}^{\infty,\si+1}.
\end{split}
\end{equation}

\subsection{Clebsch-Gordan coefficients}
Next we consider the action of $\rho_{\tau,\infty}$ in the tensor product representation $\mathcal T$, see \eqref{def:representation T}. We will need the big $q$-Laguerre polynomials, a subclass of the big $q$-Jacobi polynomials \cite{AA},\cite{KLS}. The big $q$-Laguerre polynomials are defined by
\begin{equation} \label{eq:bigqLaguerre}
L_n(x;a,b;q) = \rphis{3}{2}{q^{-n},0,x}{aq,bq}{q,q},\qquad n \in \N.
\end{equation}
For $0<a<q^{-1}$ and $b<0$ they satisfy the orthogonality relations
\[
\begin{split}
\int_{bq}^{aq} &\frac{ (x/a,x/b;q)_\infty }{ (x;q)_\infty} L_m(x;a,b;q) L_n(x;a,b;q)\, d_qx =\\
&\de_{mn}\, aq(1-q) \frac{ (q;q)_\infty \te(b/a;q) }{ (aq,bq;q)_\infty } \frac{ (q;q)_n }{ (aq,bq;q)_n} (-abq^2)^n q^{\hf n(n-1)}.
\end{split}
\]
Note that $a$ and $b$ can be interchanged, so the orthogonality relations are also valid for $a<0$ and $0 < b < q^{-1}$. From \eqref{eq:sum p<->-p} it follows that, for $p \in \N$, the polynomials $L_n$ satisfy the identity
\begin{equation}  \label{eq:n<->n-p}
(q^{1-p};q)_\infty L_n(x;q^{-p},b;q) =
\begin{cases}
\dfrac{(q^{1+p};q)_\infty (q^{-n},x;q)_p }{ (bq;q)_p }q^p L_{n-p}(xq^{p};q^{p},bq^p;q),& 0 \leq p \leq n,\\
0, & p > n,
\end{cases}
\end{equation}
which will be useful later on. We define orthonormal functions related to the big $q$-Laguerre polynomials by
\[
\bar L_n(x;a,b;q) = (-abq^2)^{-n/2} q^{-\frac14 n(n-1)} \sqrt{ \frac{|x|(x/a,x/b,aq,bq;q)_\infty (aq,bq;q)_n}{ |aq| (q,x;q)_\infty \te(b/a;q)(q;q)_n} } L_n(x;a,b;q).
\]
These functions satisfy the orthogonality relations and dual orthogonality relations
\begin{gather*}
\sum_{x \in aq^{1+\N} \cup bq^{1+\N}} \bar L_m(x;a,b;q) \bar L_n(x;a,b;q)= \de_{mn},\\
\sum_{ n \in \N} \bar L_m(x;a,b;q) \bar L_n(y;a,b;q) = \de_{xy},
\end{gather*}
provided $x,y \in aq^{1+\N} \cup bq^{1+\N}$ in the last relation. The three-term recurrence relation for $L_n$ gives the relations
\begin{equation} \label{eq:ttrbigqLaguerre}
\begin{split}
\frac{x}{bq} \bar L_n(x)= &[q^n(a/b+a+1)-aq^{2n}(1+q)]\, \bar L_n(x)\\
&-q^{n/2} \sqrt{-a/b\, (1-q^{n+1})(1-aq^{n+1})(1-bq^{n+1})}\, \bar L_{n+1}(x)\\
&- q^{(n-1)/2} \sqrt{-a/b\,  (1-q^{n})(1-aq^{n})(1-bq^{n})}\, \bar L_{n-1}(x).
\end{split}
\end{equation}
\*\\

We now consider the operator $\mathcal T(\rho_{\tau,\infty})$. Taking the appropriate limit in \eqref{eq:coproduct rho tau sigma} we find
\[
\begin{split}
\De(\rho_{\tau,\infty}) =& \rho_{\tau,\infty} \tensor 1 +q^{-1} [(1+q^2 ) \rho_{\tau,\infty} + (1-q^{2\tau})] \tensor \ga \be\\
& +i \be_{\tau+1,\infty} \de_{\tau,\infty} \tensor \de \ga - iq^{-2} \al_{\tau+1,\infty} \ga_{\tau,\infty} \tensor \be \al.
\end{split}
\]
Using Proposition \ref{prop:actions on v tau infty} and \eqref{eq:repr} we see that $\mathcal T(\rho_{\tau,\infty})$ acts `nicely' on the $\ell^2(\N)\tensor \ell^2(\N)$-basis  \mbox{$\{v_{y}^{\tau,\infty} \tensor e_n \mid y \in -q^{2\N} \cup q^{2\tau+2\N}, n \in \N\}$};
\begin{equation} \label{eq:Trhotauinfty}
\begin{split}
\mathcal T(\rho_{\tau,\infty})\, v_{y}^{\tau,\infty} \tensor e_n = &[y(1-q^{2n}-q^{2n+2}) -q^{2n}(1 -q^{2\tau})]\, v_{y}^{\tau,\infty} \tensor e_n \\
&+\ ie^{i(\psi-2\phi)} q^{n+\tau} \sqrt{ (1-q^{2n+2})(1+yq^2)(1-y q^{2-2\tau})} \, v_{yq^2}^{\tau,\infty} \tensor e_{n+1} \\
&-\ i e^{-i(\psi-2\phi)} q^{n+\tau-1} \sqrt{ (1-q^{2n})(1+y )(1-y q^{-2\tau})}\, v_{y/ q^{2}}^{\tau,\infty} \tensor e_{n-1}.
\end{split}
\end{equation}
We can identity a restriction of $\mathcal T(\rho_{\tau,\infty})$ with the Jacobi operator corresponding to the big $q^2$-Laguerre polynomials. We define, for $y\in -q^{2\Z} \cup q^{2\tau + 2\Z}$, the Clebsch-Gordan coefficients
\[
c_{x,y,n}^{\tau,\infty} = \big(\sgn(y)i\big)^n\bar L_n(-xyq^{2-2\tau} ; -y, y q^{-2\tau};q^2).
\]
\begin{prop} \label{prop:eigenvector V tau infty}
For $x \in -q^{2\N}\cup q^{2\tau+2\N}$ and $y \in -q^{2\Z}\cup q^{2\tau+2\Z}$ the vector
\[
V_{x,y}^{\tau,\infty} =
\sum_{n \in \N} e^{in(\psi-2\phi)} c_{x,y,n}^{\tau,\infty} \, v_{yq^{2n}}^{\tau,\infty} \tensor e_n,
\]
is an eigenvector of $\mathcal T(\rho_{\tau,\infty})$ for eigenvalue $x$, using the conventions $v_{\la q^{-2n}}^{\tau,\infty} = e_{-n}=0$ for $n \in \N_{\geq 1}$ and $\la \in \{-1,q^{2\tau}\}$. Furthermore,
\[
\left\{V_{x,y}^{\tau,\infty} \mid x \in -q^{2\N}\cup q^{2\tau+2\N}, y \in -q^{2\Z} \cup q^{2\tau + 2\Z}  \right\}
\]
is an orthonormal basis for $\ell^2(\N) \tensor \ell^2(\N)$, i.e.
\[
\left\langle V_{x_1,y_1}^{\tau,\infty}, V_{x_2,y_2}^{\tau,\infty} \right\rangle_{\ell^2(\N)^{\tensor 2} } = \de_{x_1x_2} \de_{y_1y_2}.
\]
\end{prop}
\begin{proof}
The proof is very similar to the proof of Proposition \ref{prop:eigenvector rho tau infty}.
We consider the action of $\mathcal T(\rho_{\tau,\infty})$ on basis vectors $v_y^{\tau,\infty}\tensor e_n$. Write $y = \la q^{2m}$ with $\la \in \{-1,q^{2\tau}\}$, $m \in \N$, and let $p = m-n$. For $p\geq 0$ the result follows from comparing \eqref{eq:Trhotauinfty} with the three-term recurrence relation \eqref{eq:ttrbigqLaguerre}. For $p< 0$ the action \eqref{eq:Trhotauinfty} of $\mathcal T(\rho_{\tau,\infty})$ actually corresponds to the three-term recurrence relation of $\bar L_m(\,\cdot\, ; q^{-2p}, -\la^2 q^{-2\tau};q^2)$, up to a phase factor. In this case we obtain
\[
\sum_{m \in \N} (\sgn(y)i)^{m-p} e^{i(m-p)(\psi-2\phi)} \bar L_m(-\la x q^{2-2\tau} ; q^{-2p}, -\la^2 q^{-2\tau};q^2) \, v_{\la q^{2m}}^{\tau,\infty} \tensor e_{m-p}
\]
as an eigenvector for eigenvalue $x$. Using the conventions $v_{\la q^{-2n}}^{\tau,\infty} = e_{-n}=0$ for $n \in \N_{\geq 1}$ and identity \eqref{eq:n<->n-p}, it follows that this eigenvector coincides with $V_{x,y}^{\tau,\infty}$.
\end{proof}
\begin{remark}
In the limit $\tau \to \infty$ `half' of the spectrum of $\rho_{\tau,\infty}$, namely the part $q^{2\tau+2\N}$, vanishes. For $x \in -q^{2\N}$ and $y \in -q^{2\Z}$ (both corresponding to the remaining part of the spectrum), we have
\[
\lim_{\tau \to \infty} L_n(-xy q^{2-2\tau};-y,yq^{-2\tau};q^2) = p_n(-x;-y;q^2),
\]
which follows directly from the $q$-hypergeometric expressions of both functions. Furthermore, in this limit the orthogonality relations for the big $q^2$-Laguerre polynomials $L_n(-xy q^{2-2\tau};-y,yq^{-2\tau};q^2)$ go over into the orthogonality relations for the Wall polynomials $p_n(-x;-y;q^2)$ (at least formally). In this way we see that the limit of a Clebsch-Gordan coefficient for $\rho_{\tau,\infty}$ is a Clebsch-Gordan coefficients for $\rho_{\infty,\infty} = -\ga \ga^*$ \eqref{eq:Clebsch-Gordan inftyinfty}, as expected:
\[
\lim_{\tau \to \infty} c^{\tau,\infty}_{x,-q^{2p},n} = c_{-x,p,n}.
\]
\end{remark}

Similarly as for $v_x^{\tau,\infty}$ we can determine the actions of $\al_{\tau,\infty}$, $\be_{\tau,\infty}$, $\ga_{\tau,\infty}$, $\de_{\tau,\infty}$ on the vectors $V_{x,y}^{\tau,\infty}$.
\begin{prop} \label{prop:T action V tau infty}
For $x \in -q^{2\N} \cup q^{2\tau + 2\N}$ and $y \in -q^{2\Z} \cup q^{2\tau+2\Z}$,
\[
\begin{split}
\mathcal T(\al_{\tau,\infty})\, V_{x,y}^{\tau,\infty} &= i e^{i(\psi-\phi)}q^\hf \sqrt{1+x }\, V_{x/q^2,y}^{\tau-1,\infty},\\
\mathcal T(\be_{\tau,\infty})\, V_{x,y}^{\tau,\infty} &= i e^{i(\phi-\psi)}q^{\tau-\hf} \sqrt{1-x q^{2-2\tau} }\, V_{x,y/q^2}^{\tau-1,\infty},\\
\mathcal T(\ga_{\tau,\infty})\, V_{x,y}^{\tau,\infty} &= i e^{i(\psi-\phi)}q^{\tau+\hf} \sqrt{1-x q^{-2\tau} }\, V_{x,yq^2}^{\tau+1,\infty}, \\
\mathcal T(\de_{\tau,\infty})\, V_{x,y}^{\tau,\infty} &= -i e^{i(\phi-\psi)} q^{-\hf} \sqrt{1+xq^2 }\, V_{xq^2,y}^{\tau+1,\infty}.
\end{split}
\]
\end{prop}
\begin{proof}
Let us calculate the action of $\be_{\tau,\infty}$. First observe that
\[
\De(\be_{\tau,\infty}) =  \be_{\tau, \infty} \tensor \de - iq^{-1}\,\al_{\tau,\infty}\tensor \be
\]
by \eqref{eq:coproducttausi} and \eqref{eq:defalbegadetausi}. Using Proposition \ref{prop:actions on v tau infty} and \eqref{eq:repr} this gives
\[
\mathcal T(\be_{\tau,\infty})\, v_y^{\tau,\infty} \tensor e_n = c_1 v_y^{\tau-1} \tensor e_{n+1} + c_2 v_{y/q^2}^{\tau-1}\tensor e_n,
\]
for certain coefficients $c_{j}$, $j=1,2$. Since $\be_{\tau,\infty} \rho_{\tau,\infty} = \rho_{\tau-1,\infty} \be_{\tau,\infty}$ by \eqref{eq:commrelation2}, we conclude that $\mathcal T(\be_{\tau,\infty}) V_{x,y}^{\tau,\infty} = C\, V_{x,y/q^2}^{\tau-1,\infty}$, where the value of $C$ still needs to be determined. To determine this value we use
\[
\De(\be_{\tau,\infty}^*) = -\De(\ga_{\tau-1,\infty}) = -\ga_{\tau-1,\infty}\tensor \al -iq\,\de_{\tau-1,\infty}\tensor \ga,
\]
see \eqref{eq:adjointstausi} and \eqref{eq:coproducttausi}, then
\[
\begin{split}
C \langle V_{x,y/q^2}^{\tau-1,\infty}, v_{y/q^2}^{\tau-1,\infty}\tensor e_0 \rangle &= -\langle V_{x,y}^{\tau,\infty}, \mathcal T(\ga_{\tau-1,\infty})\, v_{y/q^2}^{\tau-1,\infty}\tensor e_0 \rangle \\
& = e^{i(\phi-\psi)} q^{\frac12} \sqrt{1+y}\, \langle V_{x,y}^{\tau,\infty}, v_{y}^{\tau,\infty} \tensor e_0 \rangle.
\end{split}
\]
The value of $C$ now follows from the explicit expression for $c_{x,y,0}^{\tau,\infty}$. From $\be_{\tau+1,\infty} \ga_{\tau,\infty} = q \rho_{\tau,\infty}- q^{2\tau+1}$ we obtain the action of $\ga_{\tau,\infty}$. The actions of $\al_{\tau,\infty}$ and $\de_{\tau,\infty}$ are obtained in a similar way.
\end{proof}
The action of $\al_{\tau,\infty}$ implies the following contiguous relation for big $q$-Laguerre polynomials:
\[
(1-\frac{x}{bq}) L_n(x;a,bq;q) = (1-\frac1{bq})(1-aq^{n+1}) L_{n+1}(x;a,b;q) - \frac{a}{b}(1-bq) q^n L_n(x;a,b;q).
\]
The action of $\ga_{\tau,\infty}$ implies the same contiguous relation with $a$ and $b$ interchanged.\\

We can also find eigenvectors of $\mathcal T(\rho_{\infty,\si})$. Define for $x \in -q^{2\N}\cup q^{2\si+2\N}$, $y \in -q^{2\Z} \cup q^{2\si+2\Z}$, $n \in \N$, the Clebsch-Gordan coefficients by
\[
c_{x,n,y}^{\infty,\si} = \big(\sgn(y)i\big)^n\bar L_n(-xyq^{2-2\si} ; -y, y q^{-2\si};q^2),
\]
then for the vector
\[
V_{x,y}^{\infty,\si} =
\sum_{n \in \N} e^{in(\psi-2\phi)} c_{x,n,y}^{\infty,\si} \, e_n \tensor v_{yq^{2n}}^{\infty,\si},
\]
is an eigenvector of $\mathcal T(\rho_{\tau,\infty})$ for eigenvalue $x$. Furthermore,
\begin{equation} \label{eq:T action V infty si}
\begin{split}
\mathcal T(\al_{\infty,\si})\, V_{x,y}^{\infty,\si} &= i e^{i(\phi-\psi)}q^\hf \sqrt{1+x }\, V_{x/q^2,y}^{\infty,\si-1},\\
\mathcal T(\be_{\infty,\si})\, V_{x,y}^{\infty,\si} &= i q^{\si+\hf} \sqrt{1-x q^{-2\si} }\, V_{x,y/q^2}^{\infty,\si+1},\\
\mathcal T(\ga_{\infty,\si})\, V_{x,y}^{\infty,\si} &= i q^{\si-\hf} \sqrt{1-x q^{2-2\si} }\, V_{x,yq^2}^{\infty,\si-1}, \\
\mathcal T(\de_{\infty,\si})\, V_{x,y}^{\infty,\si} &= -i e^{i(\psi-\phi)} q^{-\hf} \sqrt{1+xq^2 }\, V_{xq^2,y}^{\infty,\si+1}.
\end{split}
\end{equation}

\section{$(\tau,\si)$-Clebsch-Gordan coefficients} \label{sec:tausi-CGC}
We determine the Clebsch-Gordan coefficients corresponding to (generalized) eigenvectors of $\mathcal T(\rho_{\tau,\si})$.

\subsection{Eigenvectors of $\rho_{\tau,\si}$}
We first give eigenvectors of $\pi_{\phi}(\rho_{\tau,\si})$. These eigenvectors, with a different normalization, can be found in \cite[\S6]{KoeVer}. To define the eigenvectors we need the Al-Salam--Chihara polynomials \cite{ASCh}. These polynomials are Askey-Wilson polynomials, see \cite{AW},\cite{KLS}, with two parameters equal to zero. Later on we also need the continuous dual $q$-Hahn polynomials, which are Askey-Wilson polynomials with one parameter equal to zero.

The continuous dual $q$-Hahn polynomials are polynomials in $\mu_x = \frac12( x+x^{-1})$ defined by
\begin{equation} \label{eq:ContinuousDualqHahn}
p_n(\mu_x;a,b,c|q) = a^{-n}(ab,ac;q)_n \rphis{3}{2}{q^{-n}, ax, a/x}{ ab, ac}{q,q}.
\end{equation}
The three-term recurrence relation for the continuous dual $q$-Hahn polynomials is
\begin{equation} \label{eq:3term rec CDqH}
\begin{split}
(x+x^{-1}) p_n(\mu_x) = p_{n+1}(\mu_x) &+ [abcq^{2n-1}(1+q)-q^n(a+b+c)]p_n(\mu_x) \\& + (1-q^n)(1-abq^{n-1})(1-acq^{n-1})(1-bcq^{n-1}) p_{n-1}(\mu_x).
\end{split}
\end{equation}
This shows that the continuous dual $q$-Hahn polynomials $p_n$ are symmetric in $a,b,c$, which can also be obtained from applying the $_3\varphi_2$-transformation \cite[(III.11)]{GR} to the $_3\varphi_2$-series in the definition of $p_n$.
From \eqref{eq:sum p<->-p} we find the useful identity, for $p\in \N$,
\begin{equation} \label{eq:CDqH p<->-p}
p_n(\mu_x;a,b,q^{1-p}/a;q) =
\begin{cases}
(-a)^{-p} q^{-\frac12 p(p-1)} (ax^{\pm 1};q)_p p_{n-p}(\mu_x;aq^{p},b,q/a;q), & 0 \leq p \leq n,\\
0, & p >n.
\end{cases}
\end{equation}
If all parameters are real and if $ab, ac, bc <1$, then the polynomials are orthogonal on a subset of $\R$. The orthogonality relations read
\begin{equation} \label{eq:orth rel CDqH}
\begin{split}
\int_{-1}^1 p_n(\mu_x) p_m(\mu_x) w(\mu_x;a,b,c|q) d\mu_x +  \sum_{\substack{k \in \N \\ |\al q^k|>1}} p_m(\mu_{\al q^k})&p_n(\mu_{aq^k})w(\mu_{\al q^k};a,b,c;q) \\ & =\frac{ \de_{nm} }{ (q^{n+1}, abq^{n}, acq^n, bcq^n;q)_\infty},
\end{split}
\end{equation}
where $\al$ is any of the parameters $a,b,c$. Assuming $\al=a$ for the discrete part, the weight functions are given by
\begin{gather*}
w(\mu_x;a,b,c|q) = \frac{1}{2\pi \sqrt{1-\mu_x}}   \frac{ (x^{\pm 2};q)_\infty }{ (ax^{\pm 1}, bx^{\pm 1}, cx^{\pm 1} ;q)_\infty},\qquad x\in \T,\\
w(\mu_{aq^k};a,b,c|q) = \frac{ (a^{-2};q)_\infty} { (q,ab,ac,b/a,c/a;q)_\infty} \frac{ (1-a^2q^{2k}) (a^2,ab,ac;q)_k }{ (1-a^2) (q,aq/b,aq/c;q)_k }q^{-\frac12 k (k-1)} (-a^2bc)^{-k}.
\end{gather*}
We denote by $I= I_{a,b,c;q}$ the support of the orthogonality measure, so $I$ consists of the interval $[-1,1]$ and a finite (possibly empty) discrete part. We define orthonormal functions by
\[
\bar p_n(\mu_x;a,b,c|q) =
\sqrt{w(\mu_x;a,b,c|q) (q^{n+1}, abq^{n}, acq^n, bcq^n;q)_\infty}\, p_n(\mu_x;a,b,c|q),
\]
then $\{\bar p_n\}_{n \in \N}$ is an orthonormal basis for $L^2(I)$, and the three-term recurrence relation becomes
\begin{equation} \label{eq:ttr cdqHahn}
\begin{split}
2\mu_x\,\bar p_n(\mu_x) = & \sqrt{(1-q^{n+1})(1-abq^{n})(1-acq^{n})(1-bcq^{n})} \bar p_{n+1}(\mu_x)  \\ &+ [abcq^{2n-1}(1+q)-q^n(a+b+c)]\bar p_n(\mu_x)\\& + \sqrt{(1-q^{n})(1-abq^{n-1})(1-acq^{n-1})(1-bcq^{n-1})} \bar p_{n-1}(\mu_x).
\end{split}
\end{equation}
The Al-Salam--Chihara polynomials $q_n$ are obtained from the continuous dual $q$-Hahn polynomials by setting $c=0$, i.e.~\mbox{$q_n(\mu_x;a,b|q) = p_n(\mu_x;a,b,0|q)$}. Explicit expressions for $q_n$ are
\begin{equation} \label{eq:Al-Salam--Chihara}
\begin{split}
q_n(\mu_x;a,b|q) &= a^{-n}(ab;q)_n \rphis{3}{2}{q^{-n}, ax, a/x}{ ab,0}{q,q} \\
& = x^{n} (b/x;q)_n \rphis{2}{1}{q^{-n},ax}{xq^{1-n}/b}{q,\frac{q}{bx}}.
\end{split}
\end{equation}
The three-term recurrence relation and orthogonality relations are obtained by letting $c\to 0$ in \eqref{eq:3term rec CDqH} and \eqref{eq:orth rel CDqH}. We define
\[
\bar q_n(\mu_x;a,b|q) = \lim_{c \to 0} \bar p_n(\mu_x;a,b,c|q),
\]
then $\{\bar q_n\}_{n \in \N}$ is an orthonormal basis for $L^2(I_{a,b,0;q})$.\\

To determine eigenvectors of $\rho_{\tau,\si}$ we consider the action on the vectors $v_y^{\tau,\infty}$ from Proposition \ref{prop:eigenvector rho tau infty}. First we observe from \eqref{eq:rho si tau 2} that
\[
2\rho_{\tau,\si} = (q^{1-\si-\tau}-q^{1+\si-\tau})\rho_{\tau,\infty} + q^{-\tau}\be_{\tau+1,\infty}\de_{\tau,\infty} - q^{-1-\tau}\al_{\tau+1,\infty}\ga_{\tau,\infty}.
\]
Using Propositions \ref{prop:eigenvector rho tau infty} and \ref{prop:actions on v tau infty} we obtain, for $y \in -q^{2\N} \cup q^{2\tau+2\N}$,
\begin{equation} \label{eq:actionrhotausi}
\begin{split}
2\pi_\phi(\rho_{\tau,\si})\, v_{y}^{\tau,\infty}=&\, e^{-2i\phi} \sqrt{ (1+ y q^2)( 1- y q^{2-2\tau}) }\, v_{yq^2}^{\tau,\infty}\\
& + (q^{1-\si-\tau}-q^{1+\si-\tau})y \, v_{y}^{\tau,\infty}
+ e^{2i\phi} \sqrt{ (1+ y)( 1- yq^{-2\tau}) }\, v_{y/q^2}^{\tau,\infty}.
\end{split}
\end{equation}
This operator can be matched to the Jacobi operator for the Al-Salam--Chihara polynomials.
For $y \in -q^{2\N} \cup q^{2\tau+2\N}$ we define
\[
m_{x,y}^{\tau,\si} = \bar q_n(\mu_x;-\la q^{\si-\tau+1}, \la q^{1-\si-\tau}| q^2),
\]
where $y = \la q^{2n}$ with $\la \in \{-1,q^{2\tau}\}$. We also write $I_{-\la q^{\si-\tau+1}, \la q^{1-\si-\tau},0; q^2} = I_\la^{\tau,\si}$ for the corresponding support.
\begin{prop} \label{prop:intertwiner for rho tau si}
The operator $\Te^{\tau,\si}:\ell^2(\N) \to L^2(I_{-1}^{\tau,\si})\oplus L^2(I_{q^{2\tau}}^{\tau,\si})$ defined by
\[
\Te^{\tau,\si}\,v_y^{\tau,\infty} (\mu_x) =  e^{-2in \phi} m_{x,y}^{\tau,\si}, \qquad y = \la q^{2n},
\]
is unitary and intertwines $\pi_\phi(\rho_{\tau,\si})$ with the multiplication operator $M$ on $L^2(I_{-1}^{\tau,\si})\oplus L^2(I_{q^{2\tau}}^{\tau,\si})$.
\end{prop}
Here the multiplication operator $M$ is defined by $Mf(\mu_x) = \mu_x f(\mu_x)$ almost everywhere.
\begin{proof}
The operator $\Te^{\tau,\si}$ is unitary, since it maps one orthonormal basis to another. We set $y=\la q^{2n}$ in \eqref{eq:actionrhotausi}. Comparing \eqref{eq:actionrhotausi} with the three-term recurrence relation \eqref{eq:ttr cdqHahn} with $a = -\la q^{\si-\tau+1}$, $b=\la q^{1-\si-\tau}$ and $c=0$, we see that $\Te^{\tau,\si}$ intertwines $\pi_\phi(\rho_{\tau,\si})$ with $M$.
\end{proof}
Note that $\pi_\phi(\rho_{\tau,\si})$ has continuous spectrum $[-1,1]$ with multiplicity two and (possibly empty) finite discrete spectrum
\[
\{\mu_{-q^{1-\si-\tau+2k}}\mid k \in \N, q^{1-\si-\tau+2k}>1\} \cup  \{\mu_{q^{1-\si+\tau+2k}} \mid k \in \N, q^{1-\si+\tau+2k}>1 \}.
\]
Since the spectrum is (partly) continuous, we do not have eigenvectors in general. We can, however, formulate Proposition \ref{prop:intertwiner for rho tau si} in terms of generalized eigenvectors. Let $v_x^{\tau,\si}(\la)$ denote the formal sum
\[
v_{x}^{\tau,\si;\la} = \sum_{n \in \N } e^{-2in \phi} m_{x,\la q^{2n}}^{\tau,\si} \, v_{\la q^{2n}}^{\tau,\infty}, \qquad \la = -1, q^{2\tau}.
\]
For $\mu_x \in [-1,1]$, i.e., $x \in \T$, this can be considered as a generalized eigenvector
of $\pi_\phi(\rho_{\tau,\infty})$ for eigenvalue $\mu_x$. For $\mu_x$ in the discrete spectrum this is a genuine eigenvector.

\begin{remark}
In this paper we will use generalized eigenvectors as if they were genuine eigenvectors, in particular we consider inner products of generalized eigenvectors with other vectors $v$. In general the inner products do not exist, but for specific vectors $v$ the inner products do make sense. For example,
\[
\langle v_{x}^{\tau,\si;\la}, v_{\la q^{2n}}^{\tau,\infty} \rangle = e^{-2in \phi} m_{x,\la q^{2n}}^{\tau,\si} = \Te^{\tau,\si}\,v_{\la q^{2n}}^{\tau,\infty} (\mu_x).
\]
\end{remark}

Similar as in Proposition \ref{prop:actions on v tau infty} we can determine actions of $\al_{\tau,\si}, \be_{\tau,\si}, \ga_{\tau,\si}, \de_{\tau,\si}$ on the (generalized) eigenvectors. We do not need these actions later on, so we omit the details. The result is
\[
\begin{split}
\pi_\phi(\al_{\tau,\si})\, v_x^{\tau,\si} &= i e^{-i\phi} q^{\tau+ \si-\hf} \sqrt{(1+ q^{1-\si-\tau}x) (1+ q^{1-\si-\tau}/x)} \, v_x^{\tau-1,\si-1},\\
\pi_\phi(\be_{\tau,\si})\, v_x^{\tau,\si} &= i e^{-i\phi} q^{\tau-\hf} \sqrt{(1- q^{\si-\tau+1} x) (1- q^{\si-\tau+1}/x)} \, v_x^{\tau-1,\si+1},\\
\pi_\phi(\ga_{\tau,\si})\, v_x^{\tau,\si} &= i e^{i\phi} q^{\tau+ \hf} \sqrt{(1- q^{\si-\tau-1} x)(1- q^{\si-\tau-1}/x)}\, v_x^{\tau+1,\si-1},\\
\pi_\phi(\de_{\tau,\si})\, v_x^{\tau,\si} &= -i e^{i\phi} q^{\tau+ \si+\hf} \sqrt{(1+ q^{-1-\si-\tau} x) (1+ q^{-1-\si-\tau} /x)}\, v_x^{\tau+1,\si+1}.
\end{split}
\]
Alternatively, we can  reformulate these as actions on the function $m_{x,y}^{\tau,\si}$ using the operator $\Te^{\tau,\si}$ from Proposition \ref{prop:intertwiner for rho tau si}. Define $\Te = \bigoplus_{m,n \in \Z} \Te^{\tau+m,\si+n}$, then
\[
\pi'_\phi(X) m_{x,y}^{\tau,\si} = \Te \pi_\phi(X) \Te^* m_{x,y}^{\tau,\si}, \qquad X \in \mathcal A_q,
\]
gives an action of $\mathcal A_q^{\tau,\si}$ on $\bigoplus_{m,n \in \Z} L^2(I_{-1}^{\tau+m,\si+n}) \oplus L^2(I_{q^{2\tau+2m}}^{\tau+m,\si+n})$.

For the genuine eigenvectors $v_{x}^{\tau,\si}$, i.e.~with $\mu_x$ in the discrete spectrum, we can rewrite the actions as `dynamic' raising and lowering operators. Indeed, let $x_k = -q^{1-\tau-\si+2k}$ such that $\mu_{x_k}$ is in the discrete spectrum of $\pi_\phi(\rho_{\tau,\si})$. We set $\hat v_k^{\tau,\si} = v_{x_k}^{\tau,\si}$, then
\[
\begin{split}
\pi_\phi(\al_{\tau,\si})\, \hat v_k^{\tau,\si} &= i e^{-i\phi} q^{\tau+ \si-\hf} \sqrt{(1 -  q^{2-2\si-2\tau+2k}) (1- q^{-2k}) } \, \hat v_{k-1}^{\tau-1,\si-1},\\
\pi_\phi(\be_{\tau,\si})\, \hat v_k^{\tau,\si} &= i e^{-i\phi} q^{\tau-\hf} \sqrt{(1 + q^{2-2\tau+2k}) (1 + q^{2\si-2k})} \, \hat v_k^{\tau-1,\si+1},\\
\pi_\phi(\ga_{\tau,\si})\, \hat v_k^{\tau,\si} &= i e^{i\phi} q^{\tau+ \hf} \sqrt{(1+ q^{-2\tau+2k})(1+ q^{2\si-2-2k})}\, \hat v_k^{\tau+1,\si-1},\\
\pi_\phi(\de_{\tau,\si})\, \hat v_k^{\tau,\si} &= -i e^{i\phi} q^{\tau+ \si+\hf} \sqrt{(1- q^{-2\si-2\tau+2k}) (1- q^{-2+2k})}\, \hat v_{k+1}^{\tau+1,\si+1}.
\end{split}
\]
There is a similar result for the vectors $v_{x_k}^{\tau,\si}$, with $x_k = q^{1-\si+\tau+2k}$ such that $\mu_{x_k}$ is in the discrete spectrum.

\subsection{Clebsch-Gordan coefficients} \label{ssec:CGCtausi}
To determine Clebsch-Gordan coefficients, we diagonalize the action of $\mathcal T(\rho_{\tau,\si})$ on a suitable basis. From letting $\la \to \infty$ in \eqref{eq:coproduct rho tau sigma} we obtain
\begin{equation} \label{eq:De(rhotausi)}
\begin{split}
\De(\rho_{\tau,\si})=& \hf q^{-\tau-\si-1} \Big( q^{-1}\, \al_{\tau+1,\infty} \ga_{\tau,\infty} \tensor \al_{\infty, \si+1} \be_{\infty,\si}
+ q\, \be_{\tau+1,\infty} \de_{\tau, \infty} \tensor \ga_{\infty, \si+1} \de_{\infty,\si}\\
& + q^2(1+q^2)\, \rho_{\tau,\infty} \tensor \rho_{\infty, \si} + q^2(1-q^{2\si})\, \rho_{\tau,\infty} \tensor 1 + q^2(1-q^{2\tau})\, 1 \tensor \rho_{\infty, \si} \Big)
\end{split}
\end{equation}
We now let $\mathcal T(\rho_{\tau,\si})$ act on the $\ell^2(\N) \tensor \ell^2(\N)$-basis  $\{v_y^{\tau,\infty} \tensor v_z^{\infty,\si} \mid y \in -q^{2\N}\cup q^{2\tau+2\N}, z \in -q^{2\N}\cup q^{2\si+2\N}\}$, see Propositions \ref{prop:eigenvector rho tau infty}, \ref{prop:actions on v tau infty} and \eqref{eq:actions on v infty si}. This gives
\[
\begin{split}
\mathcal T(\rho_{\tau,\si})\, &v_{y}^{\tau,\infty} \tensor v_{z}^{\infty,\si} = \\
&\hf \Big[yz q^{1-\tau -\si}(1+ q^2) + y q^{1-\tau -\si} (1- q^{2\si})  + z q^{1-\tau -\si}(1-q^{2\tau})\Big]\,  v_{y}^{\tau,\infty} \tensor v_{z}^{\infty,\si}\\
+& \hf e^{2i(\phi - \psi)} \sqrt{ (1+ y ) (1- y q^{-2\tau} ) ( 1+ z) (1-z q^{-2\si} ) } \, v_{y/q^2}^{\tau,\infty} \tensor v_{z/q^2}^{\infty,\si} \\
+& \hf e^{-2i(\phi - \psi)} \sqrt{ (1+ y q^{2} ) (1- y q^{2-2\tau} ) ( 1+ z q^{2} ) (1-z q^{2-2\si} ) } \, v_{y q^{2}}^{\tau,\infty} \tensor v_{z q^{2} }^{\infty,\si}.
\end{split}
\]
We can identify this with the Jacobi operator corresponding to continuous dual $q^2$-Hahn polynomials. We define, for $y \in -q^{2\N}\cup q^{2\tau+2\N}$ and $z \in -q^{2\Z} \cup q^{2\si +2\Z}$,
\[
c_{x,y,z}^{\tau,\si} = \bar p_n(\mu_x;-\la q^{1-\tau+\si}, \la q^{1-\tau-\si}, z q^{1+\tau-\si}/\la |q^2), \qquad y = \la q^{2n}, \ \la \in \{-1,q^{2\tau}\}.
\]
For the corresponding support we write $I_{-\la q^{1-\tau+\si}, \la q^{1-\tau-\si}, z q^{1+\tau-\si}/\la ;q^2}  = I^{\tau,\si}_{\la,z}$. Let us remark that if $z$ is too large the pairwise product of the continuous dual $q$-Hahn parameters is not smaller than $1$, so at first sight it seems that the orthogonality relations for the continuous dual $q$-Hahn polynomials are not valid. If this is the case we use \eqref{eq:CDqH p<->-p} so that all conditions for orthogonality are again satisfied.
\begin{prop}
Let $z \in -q^{2\Z} \cup q^{2\si +2\Z}$ and let $\mathcal H_z$ be the subspace of $\ell^2(\N) \tensor \ell^2(\N)$ given by
\[
\mathcal H_z= \overline{\mathrm{span}}\left\{ v_{\la q^{2n}}^{\tau,\infty} \tensor v_{z q^{2n}}^{\infty,\si}\mid n \in \N,\la \in \{-1,q^{2\tau}\}\right \}.
\]
Then the operator $\Ups : \mathcal H_z \to L^2(I^{\tau,\si}_{-1,z}) \oplus L^2(I^{\tau,\si}_{q^{2\tau},z})$ defined by
\[
\Ups v_{\la q^{2n}}^{\tau,\infty} \tensor v_{z q^{2n}}^{\infty,\si} (\mu_x)= e^{2in(\psi-\phi)} c_{x,\la q^{2n},z}^{\tau,\si},\qquad \la \in \{-1, q^{2\tau}\},
\]
is unitary and intertwines $\mathcal T(\rho_{\tau,\si})|_{\mathcal H_z}$ with the multiplication operator $M$ on $L^2(I^{\tau,\si}_{-1,z}) \oplus L^2(I^{\tau,\si}_{q^{2\tau},z})$.
\end{prop}
Similar as before, we use here the convention $v_{\nu q^{-2k}}^{\infty,\si} = 0$ for $k \in \N_{\geq 1}$ and $\nu \in \{-1,q^{2\si}\}$. Observe that $\mathcal T(\rho_{\tau,\si})|_{\mathcal H_z}$ has $[-1,1]$ as continuous spectrum, and the points in the discrete spectrum $\Sigma^{\tau,\si}$ are of the form $-\mu_{q^{1-\tau-\si+2k}}$ or $\mu_{q^{1+\tau-\si+2k}}$ for some integer $k$.
\begin{remark}
Recall that $\rho_{\tau,\infty}= \lim_{\si\to \infty}2q^{\tau+\si-1} \rho_{\tau,\si}$. We see that in the limit $\si \to \infty$ the continuous spectrum vanishes. Furthermore, the number of points in the discrete spectrum increases as $\si$ increases, and then we see that
\[
\lim_{\si\to \infty} 2q^{\tau+\si-1} \Sigma^{\tau,\si}=-q^{2\N} \cup q^{2\tau+2\N},
\]
as expected from Proposition \ref{prop:eigenvector V tau infty}. Furthermore, for $z \in -q^{2\Z}$ and $x_k = \nu q^{1-\tau+\si+2k}$ with \mbox{$\nu \in \{-1,q^{2\tau}\}$}, we have
\[
\lim_{\si \to \infty} \rphis{3}{2}{q^{-2n}, -\la q^{1-\tau+\si}x_k , -\la q^{1-\tau+\si}/x_k}{-\la^2 q^{2-2\tau}, -z \la q^2}{q^2,q^2} = \rphis{3}{2}{q^{-2n}, -\la \nu q^{2-2\tau+2k} ,0}{-\la^2 q^{2-2\tau}, -z \la q^2}{q^2,q^2},
\]
which corresponds to the limit from the continuous dual $q^2$-Hahn polynomials to the big $q^2$-Laguerre polynomials. To be precise, this gives the big $q^2$-Laguerre polynomials from the previous section in case $p<0$, see the proof of Proposition \ref{prop:eigenvector V tau infty}. From this we  obtain
\[
\lim_{\si \to \infty} c_{x_k,\la q^{2n},z}^{\tau,\si} = c^{\tau,\infty}_{\nu q^{2k}, -z\la, n}.
\]
\end{remark}

In terms of generalized eigenvectors, the vectors
\begin{equation} \label{eq:eigenvector Vxztausi}
V_{x,z}^{\tau,\si;\la} =\sum_{n \in \N } e^{2in(\psi-\phi)} c_{x,\la q^{2n},z}^{\tau,\si}\, v_{\la q^{2n}}^{\tau,\infty} \tensor v_{zq^{2n}}^{\infty,\si}, \qquad \la = -1, q^{2\tau},
\end{equation}
are both generalized eigenvectors of $\mathcal T(\rho_{\tau,\si})$ for eigenvalue $\mu_x$. Alternatively, we may write
\[
v_{y}^{\tau,\infty} \tensor v_{z}^{\infty,\si} = \int_{-1}^1 e^{- 2in(\psi-\phi)} c_{x,y,zq^{-2n}}^{\tau,\si}\, V_{x,zq^{-2n}}^{\tau,\si;\la} d\mu_x + \sum_k e^{- 2in(\psi-\phi)} c_{x_k,y,zq^{-2n}}^{\tau,\si}\, V_{x_k,z}^{\tau,\si;\la},
\]
for $y = \la q^{2n}$ and $\mu_{x_k}$ is in the discrete spectrum. The actions of the generators of $\mathcal A_q^{\tau,\si}$ are:
\[
\begin{split}
\mathcal T(\al_{\tau,\si}) V_{x,z}^{\tau,\si} & = - e^{i(\psi-\phi)} q^{\tau+\si-\hf}\sqrt{(1+q^{1-\tau-\si}x)(1+q^{1-\tau-\si}/x)}\,  V_{x,z}^{\tau-1,\si-1},\\
\mathcal T(\be_{\tau,\si}) V_{x,z}^{\tau,\si} & = e^{i(\psi-\phi)}q^{\tau-\hf} \sqrt{( 1-q^{1-\tau+\si}x)(1-q^{1-\tau+\si}/x)}\, V_{x,zq^2}^{\tau-1,\si+1},\\
\mathcal T(\ga_{\tau,\si}) V_{x,z}^{\tau,\si} & = -e^{i(\phi-\psi)}q^{\tau+\hf} \sqrt{(1-q^{-1-\tau+\si}x)(1-q^{-1-\tau+\si}/x)}\, V_{x,z/q^2}^{\tau+1,\si-1},\\
\mathcal T(\de_{\tau,\si}) V_{x,z}^{\tau,\si} & = -e^{i(\phi-\psi)} q^{\tau+\si+\hf} \sqrt{(1+q^{-1-\tau-\si}x)(1+q^{-1-\tau-\si}/x)}\,  V_{x,z}^{\tau+1,\si+1}.
\end{split}
\]

\begin{remark} \label{rem:si<->tau}
By symmetry in $\si$ and $\tau$, we may also define, for $y \in -q^{2\Z} \cup q^{2\tau +2\Z}$, the subspace $\widetilde{ \mathcal H}_y$ of $\ell^2(\N) \tensor \ell^2(\N)$ by
\[
\widetilde{\mathcal H}_y= \overline{\mathrm{span}}\left\{ v_{y q^{2n}}^{\tau,\infty} \tensor v_{\xi q^{2n}}^{\infty,\si}\mid n \in \N,\xi \in \{-1,q^{2\si}\}\right \}.
\]
Then the operator $\widetilde \Ups : \widetilde{\mathcal H}_y \to L^2(I^{\si,\tau}_{-1,y}) \oplus L^2(I^{\si,\tau}_{q^{2\si},y})$ defined by
\[
\widetilde \Ups v_{y q^{2n}}^{\tau,\infty} \tensor v_{\xi q^{2n}}^{\infty,\si} (\mu_x)= e^{2in(\psi-\phi)} c_{x,\xi q^{2n},y}^{\si,\tau},\qquad \xi \in \{-1, q^{2\si}\},
\]
is unitary and intertwines $\mathcal T(\rho_{\tau,\si})|_{\widetilde{\mathcal H}_y}$ with $M$. So, for $\xi = -1, q^{2\si}$, the vectors
\[
\widetilde V_{x,y}^{\tau,\si;\xi} =\sum_{n \in \N } e^{2in(\psi-\phi)} c_{x,\xi q^{2n},y}^{\si,\tau}\, v_{y q^{2n}}^{\tau,\infty} \tensor v_{\xi q^{2n}}^{\infty,\si}
\]
are both generalized eigenvectors of $\mathcal T(\rho_{\tau,\si})$ for eigenvalue $\mu_x$. These are related to the generalized eigenvectors \eqref{eq:eigenvector Vxztausi} by
\begin{equation} \label{eq:tildeV<->V}
\begin{pmatrix}
\widetilde V_{x,\la q^{2k}}^{\tau,\si;-1} \\\widetilde V_{x,\la q^{2k}}^{\tau,\si;q^{2\si}}
\end{pmatrix}
=
\begin{pmatrix}
\de_{\la,-1} \de_{\xi,-1} & \de_{\la,q^{2\tau}} \de_{\xi,-1} \\
\de_{\la,-1} \de_{\xi,q^{2\si}} & \de_{\la, q^{2\tau}} \de_{\xi,q^{2\si}}
\end{pmatrix}
\begin{pmatrix}
V_{x,\xi q^{-2k}}^{\tau,\si;-1} \\ V_{x,\xi q^{-2k}}^{\tau,\si;q^{2\tau}}
\end{pmatrix}.
\end{equation}
\end{remark}

\section{Coupling coefficients for two-fold tensor products} \label{sec:twofold}
In this section we determine coupling coefficients between different (generalized) eigenvectors of $\mathcal T(\rho_{\tau,\si})$. For simplicity we take the representation labels $\phi$ and $\psi$ both equal to $0$. Note that there is hardly any loss of generality in doing so, because the representation labels only occur in phase factors in eigenvectors in the previous sections. First we introduce a class of functions that we will need.

\subsection{Al-Salam--Carlitz II functions} \label{ssec:ACSfunctions}
The Al-Salam--Carlitz II polynomials $P_n$ and related functions $Q_n$, see \cite{Groen}, are defined by
\begin{equation} \label{eq:defPn}
P_n(x)=P_n(x;c,d;q) = (-c)^n q^{-\frac12 n(n-1)} \rphis{2}{0}{q^{-n},cx}{\mhyphen}{q,\frac{dq^n}{c}},
\qquad n \in \N,
\end{equation}
and
\begin{equation} \label{eq:defQn}
\begin{split}
Q_n(x) = Q_n(x;c,d;z_-,z_+;q) =  & (-d)^{n} q^{-\frac12n(n+1)} \frac{(cx;q)_\infty \te(dz_-,dz_+;q) }{ (q/dx, q^{n+1}x/dz_-z_+;q)_\infty} \\ &\times \rphis{1}{1}{q/cx}{dz_-z_+q^{-n}/x}{q,\frac{cz_-z_+q^{-n}}{x}}, \qquad \qquad n \in \Z.
\end{split}
\end{equation}
Both $P_n$ and $Q_n$ are symmetric in the parameters $c$ and $d$, and $Q_n$ is also symmetric in $z_-$ and $z_+$. For notational convenience we also define $P_n =0$ for $n \in -\N_{\geq 1}$.

Let $\overline c = d$, $z_+>0$ and $z_-<0$, then the functions $P_n$ and $Q_n$ satisfy the orthogonality relations
\begin{equation} \label{eq:orthogonalityPQ}
\begin{split}
\frac{1}{1-q}\int_{\infty(z_-)}^{\infty(z_+)} P_m(x) P_n(x) w(x)\, d_q x &= \de_{mn} h_n^P,\\
\frac{1}{1-q}\int_{\infty(z_-)}^{\infty(z_+)} Q_m(x) Q_n(x) w(x)\, d_q x &= \de_{mn} h_n^Q,\\
\frac{1}{1-q}\int_{\infty(z_-)}^{\infty(z_+)} P_m(x) Q_n(x) w(x)\, d_q x &=0,
\end{split}
\end{equation}
where $m,n \in \Z$ and
\begin{gather*}
w(x) = w(x;a,b;q) = \frac{1}{(cx,dx;q)_\infty},\\
h_n^P=h_n^P(c,d;z_-,z_+;q) =
\left\{
\begin{aligned}
&z_+(q;q)_n(cd)^{n} q^{-n^2} \frac{(q;q)_\infty\te(z_-/z_+,cdz_-z_+;q)}{\te(c z_-,dz_-,cz_+,dz_+;q)}, & n \geq 0,\\
&0, & n <0
\end{aligned} \right. \\
h_n^Q=h_n^Q(c,d;z_-,z_+;q) = z_+(-z_-z_+)^{-n}q^{-\frac12 n(n+1)} (cdz_-z_+q^{-n-1};q)_\infty  (q;q)_\infty^2 \te(z_-/z_+;q).
\end{gather*}
Moreover, $\{P_n\}_{n \in \N} \cup \{Q_n\}_{n \in \Z}$ is an orthogonal basis for the corresponding $L^2$-space, which implies the dual orthogonality relations
\[
\sum_{n \in \N} P_n(x)P_n(y) \frac{ 1}{h_n^P} + \sum_{n \in \Z} Q_n(x)Q_n(y) \frac{ 1}{h_n^Q} = \frac{\de_{xy}}{|x|w(x)},
\]
for $x,y \in z_-q^{\Z} \cup z_+ q^\Z$.

The following $q$-integral evaluation formulas turn out to be useful. The proof is given in the appendix.
\begin{prop} \label{prop:integralPQ}
For $n \in \N$,
\[
\begin{split}
\int_{z_-}^{z_+}& P_n(x) \frac{ (qx/z_-,qx/z_+;q)_\infty}{(cx,dx;q)_\infty} d_qx \\
&= (1-q)z_+(-d)^n q^{-\frac12n(n-1)} \frac{ (q,cdz_-z_+q^n;q)_\infty \te(z_-/z_+;q)}{ (cz_-q^n,dz_-,cz_+, dz_+;q)_\infty} \rphis{2}{1}{q^{-n}, dz_+ }{q^{1-n}/cz_-}{q, \frac{q}{dz_-}},\\
\int_{z_-}^{z_+}& Q_n(x) \frac{ (qx/z_-,qx/z_+;q)_\infty}{(cx,dx;q)_\infty} d_qx \\
& =(1-q)z_+ (c/q)^n q^{-\frac12 n(n-1)} (q/cz_+;q)_n (q,q^{n+1};q)_\infty \te(z_-/z_+;q) \rphis{2}{1}{q^{-n}, q/dz_-}{cz_+q^{-n}}{q,dz_+}.
\end{split}
\]
Furthermore, for $n \in -\N_{\geq 1}$,
\[
\int_{z_-}^{z_+} Q_n(x) \frac{ (qx/z_-,qx/z_+;q)_\infty}{(cx,dx;q)_\infty} d_qx = 0.
\]
\end{prop}

\textbf{A limit case.} We also need a limit case of the Al-Salam--Carlitz II functions.
We let $c \to q/z_+$ and $d \to 0$. Clearly the polynomials $P_n$ vanish in this limit. For the limit of $Q_n(x)$ we distinguish between $x \in z_+q^\Z$ and $x \in z_- q^\Z$. Using \eqref{eq:te-product} we obtain
\[
\begin{split}
Q_n(z_+q^k) = (z_-)^{-n} \left(\frac{z_+}{z_-}\right)^k q^{k(n+k)} (cz_+q^k,dz_+q^k,dz_-q^{-n-k};q)_\infty \rphis{1}{1}{q^{1-k}/cz_+}{dz_- q^{-n-k}}{q,cz_- q^{-n-k}},
\end{split}
\]
so that
\[
\lim_{\substack{c \to q/z_+\\d \to 0}} Q_n(z_+q^k) = (z_-)^{-n} \left(\frac{z_+}{z_-}\right)^k q^{k(n+k)} (q^{k+1};q)_\infty \rphis{1}{1}{q^{-k}}{0}{q,\frac{z_-q^{1-n-k}}{z_+}},
\]
which equals zero if $k \in -\N_{\geq 1}$. Reversing the order of summation in the $_1\varphi_1$-series,
\[
\rphis{1}{1}{q^{-k}}{0}{zq^{1-k}} = (zq^{-k})^k \rphis{1}{1}{q^{-k}}{0}{\frac{q^{1+k}}{z}},
\]
we can now recognize the limit of $Q_n$ as a Stieltjes-Wigert polynomial, which is defined by
\begin{equation} \label{eq:Stieltjes-Wigert}
S_k(x;q) = \frac{1}{(q;q)_k}\rphis{1}{1}{q^{-k}}{0}{-xq^{1+k}},
\end{equation}
so that
\[
\lim_{\substack{c \to q/z_+\\d \to 0}} Q_n(z_+q^k) = \frac{(q;q)_\infty}{(z_-)^{n}} S_k(-z_+q^n/z_-;q).
\]
In the same way we can express the limit of $Q_n(z_-q^k)$ in terms of the functions
\begin{equation} \label{eq:SW-qBessel}
M_k^{(t)}(x;q) = \frac{1}{(q;q)_\infty} \rphis{1}{1}{-tq^{-k}}{0}{q,\frac{xq^{k+1}}{t}},
\end{equation}
which are closely related to Jackson's second $q$-Bessel functions. We have
\[
\begin{split}
\lim_{\substack{c \to q/z_+\\d \to 0}} Q_n(z_-q^k) &= (z_+)^{-n} \left( \frac{z_-}{z_+} \right)^k q^{k(k+n)} (z_-q^{k+1}/z_+;q)_\infty \rphis{1}{1}{ z_+ q^{-k}/z_-}{0}{q;q^{1-n-k}} \\
& = (z_-)^{-n} (z_-q^{k+1}/z_+;q)_\infty \rphis{1}{1}{ z_+ q^{-k}/z_-}{0}{q;q^{1+n+k}} \\
& = (z_-)^{-n} (q,z_-q^{k+1}/z_+;q)_\infty M_k^{(-z_+/z_-)}(-z_+q^{n}/z_-;q).
\end{split}
\]
From the orthogonality relations for $Q_n$ we obtain (formally)
\[
\begin{split}
&\sum_{k \in \N} S_k\left(-\frac{z_+q^m}{z_-};q\right) S_k\left(-\frac{z_+q^n}{z_-};q\right) \frac{ z_+ q^k }{(q^{k+1};q)_\infty}\\
-& \sum_{k \in \Z} M_k^{(-z_+/z_-)}\left(-\frac{z_+q^m}{z_-};q\right) M_k^{(-z_+/z_-)}\left(-\frac{z_+q^n}{z_-};q\right) z_-q^k (z_-q^{k+1}/z_+;q)_\infty \\
&= \de_{mn}\, z_+ \left(-\frac{z_-}{z_+} \right)^n q^{-\frac12n(n+1)}  \te(z_-/z_+;q) ,
\end{split}
\]
and from the dual orthogonality relations we find
\begin{gather*}
\sum_{n \in \Z} S_k\left(-\frac{z_+q^n}{z_-};q\right) M_l\left(-\frac{z_+q^n}{z_-};q\right) \left(-\frac{z_+}{z_-} \right)^n q^{\frac12n(n+1)}  = 0, \\
\sum_{n \in \Z} S_k\left(-\frac{z_+q^n}{z_-};q\right) S_l\left(-\frac{z_+q^n}{z_-};q\right) \left(-\frac{z_+}{z_-} \right)^n q^{\frac12n(n+1)} = \de_{kl}\, q^{-k} (q^{k+1};q)_\infty \te(z_-/z_+;q),\\
\begin{split}
\sum_{n \in \Z} M_k^{(-z_+/z_-)}\left(-\frac{z_+q^n}{z_-};q\right) M_l^{(-z_+/z_-)}\left(-\frac{z_+q^n}{z_-};q\right) &\left(-\frac{z_+}{z_-} \right)^n q^{\frac12n(n+1)} = \\ &\de_{kl}\, \left(-\frac{z_+}{z_-}\right) \frac{ \te(z_-/z_+;q) }{ q^k  (z_-q^{k+1}/z_+;q)_\infty }.
\end{split}
\end{gather*}
These relations are proved by Christiansen and Koelink in \cite[Theorem 3.5]{ChK}.

\subsection{Coupling coefficients}
In Section \ref{ssec:CGCtausi} we diagonalized $\mathcal T(\rho_{\tau,\si})$ by considering the action on basis elements $v_{y}^{\tau,\infty} \tensor v_{z}^{\infty,\si}$. We can also diagonalize $\mathcal T(\rho_{\tau,\si})$ using the actions of $\al_{\tau,\infty}$, $\be_{\tau,\infty}$, $\ga_{\tau,\infty}$, $\de_{\tau,\infty}$, on the orthonormal basis $\{V_{y,z}^{\tau,\infty} \mid y \in -q^{2\N}\cup q^{2\tau+2\N},\ z \in -q^{2\Z}\cup q^{2\tau+2\Z}\}$ of $\ell^2(\N)\tensor \ell^2(\N)$, see Propositions \ref{prop:eigenvector V tau infty} and \ref{prop:T action V tau infty}. We find
\[
\begin{split}
\mathcal T&(\rho_{\tau,\infty}) V_{y,z}^{\tau,\infty} = \\
&\sqrt{(1+yq^2)(1-yq^{2-2\tau})} V_{yq^2,z/q^2}^{\tau,\infty} + (q^{1-\si-\tau}-q^{1+\si-\tau})y V_{y,z}^{\tau,\infty} + \sqrt{(1+y)(1-yq^{-2\tau})} V_{y/q^2,zq^2}^{\tau,\infty}.
\end{split}
\]
Comparing this with Proposition \ref{prop:intertwiner for rho tau si} we obtain the following result.
\begin{prop} \label{prop:intertwiner vartheta}
Define for $z \in -q^{2\Z}\cup q^{2\tau+2\Z}$ the subspace
\[
\mathcal H_z' = \overline{\mathrm{span}}\left\{ V_{\nu q^{2n},zq^{-2n}}^{\tau,\infty} \mid n \in \N,\ \nu \in \{-1,q^{2\tau}\} \right\} \subset \ell^2(\N) \tensor \ell^2(\N),
\]
then the unitary operator $\vartheta :  \mathcal H_z' \to L^2(I_{-1}^{\tau,\si}) \oplus L^2(I_{q^{2\tau}}^{\tau,\si})$ defined by
\[
\vartheta V_{\nu q^{2n},z q^{-2n} }^{\tau,\infty} (x) = m_{x,\nu q^{2n}}^{\tau,\si},
\]
intertwines $\mathcal T(\rho_{\tau,\si})|_{\mathcal H_z'}$ with multiplication operator $M$ on $L^2(I_{-1}^{\tau,\si}) \oplus L^2(I_{q^{2\tau}}^{\tau,\si})$.
\end{prop}

In terms of generalized eigenvectors, Proposition \ref{prop:intertwiner vartheta} says that
\begin{equation} \label{eq:eigenvectorUtausi}
U_{x,z}^{\tau,\si;\nu}=\sum_{n \in \N}  m_{x,\nu q^{2n}}^{\tau,\si} V_{\nu q^{2n},zq^{-2n}}^{\tau,\infty},\qquad \nu =-1,q^{2\tau},
\end{equation}
are generalized eigenvectors of $\mathcal T(\rho_{\tau,\si})$ for eigenvalue $\mu_x$, $x \in \T$. For $\mu_x$ in the discrete spectrum, these are genuine eigenvectors.

From here on  we assume $x \in \T$ in this section. Since $V_{x,y}^{\tau,\si}$ defined by \eqref{eq:eigenvector Vxztausi} are also generalized eigenvectors of $\mathcal T(\rho_{\tau,\si})$ for eigenvalue $\mu_x$ there exists a (formal) expansion
\begin{equation} \label{eq:U=sum SV}
\begin{pmatrix}
U_{x,z}^{\tau,\si;-1} \\
U_{x,z}^{\tau,\si;q^{2\tau}}
\end{pmatrix}
= \sum_{y \in -q^{2\Z} \cup q^{2\si + 2\Z}} S_{x;y,z}^{\tau,\si}
\begin{pmatrix}
V_{x,y}^{\tau,\si;-1} \\
V_{x,y}^{\tau,\si;q^{2\tau}}
\end{pmatrix}
\end{equation}
where the coupling coefficients $S_{x,y,z}^{\tau,\si}$ are $2\times 2$-matrices;
\[
S_{x;y,z}^{\tau,\si} =\left(\big(S_{x;y,z}^{\tau,\si}\big)_{\nu,\la}\right)_{\nu,\la \in \{-1,q^{2\tau}\}}.
\]
Alternatively, $S_{x;y,z}^{\tau,\si}$ is the kernel of the unitary integral operator $S:L^2(I_{-1}^{\tau,\si}) \oplus L^2(I_{q^{2\tau}}^{\tau,\si}) \to L^2(I^{\tau,\si}_{-1,z}) \oplus L^2(I^{\tau,\si}_{q^{2\tau},z})$ that satisfies $(S \circ \vartheta f)(x) = \Upsilon f(x)$ (almost everywhere), where $\vartheta$ and $\Upsilon$ are the operators defined in Propositions \ref{prop:intertwiner vartheta} and \ref{prop:eigenvector V tau infty}. The unitarity of $S$ is equivalent to the matrix orthogonality relations
\begin{equation} \label{eq:orthogonalityS}
\begin{split}
\sum_{ y \in -q^{2\Z} \cup q^{2\si+2\Z} }S_{x;y,z_1}^{\tau,\si} (S_{x;y,z_2}^{\tau,\si} )^*& = \de_{z_1z_2} I, \\
\sum_{ z \in -q^{2\Z} \cup q^{2\tau+2\Z}} (S_{x;y_1,z}^{\tau,\si})^* S_{x;y_2,z}^{\tau,\si} & = \de_{y_1y_2} I.
\end{split}
\end{equation}
In order to determine explicit expressions for the matrix coefficients of $S_{x;y,z}^{\tau,\si}$ we use the following result.
\begin{lemma}
Let $\xi \in \{-1,q^{2\tau}\}$, $k \in \Z$ and $m,n \in \N$. Then for $\la =\xi$,
\begin{equation} \label{eq:Sxyz1}
\sum_{y \in -q^{2\Z} \cup q^{2\si + 2\Z}} \big(S_{x;y,\xi q^{2k}}^{\tau,\si}\big)_{\nu,\la} c_{x,\la q^{2n},y}^{\tau,\si} m_{yq^{2n},m}^{\infty,\si} =
m_{x,\nu q^{2m-2n+2k}}^{\tau,\si} c_{\nu q^{2m-2n+2k},\la  q^{2n-2m},m}^{\tau,\infty}.
\end{equation}
Furthermore, $\big(S_{x;y,\xi q^{2k}}^{\tau,\si}\big)_{\nu,\la} = 0 $ for $\la \neq \xi$.
\end{lemma}
\begin{proof}
This follows from \eqref{eq:U=sum SV} by taking inner products with the orthonormal basis vectors \mbox{$v_{\la q^{2n}}^{\tau,\infty} \tensor e_m$}.
\end{proof}

\begin{thm} \label{thm:Stausi}
Let $x \in \T$, $y \in -q^{2\Z}\cup q^{2\si+2\Z}$, $\la,\nu, \xi \in \{-1,q^{2\tau}\}$, and $z = \xi q^{2k}$ with $k \in \Z$.
\begin{enumerate}[(i)]
\item If $\xi \neq \la$, then $S_{x;y,z}^{\tau,\si;\nu,\la} = 0$.
\item If $\xi = \la$, then
\[
\big(S_{x;y,z}^{\tau,\si}\big)_{\nu,\la} =
\begin{cases}
\displaystyle (\sgn(\nu))^k \sqrt{ \frac{|y| w(y;c,d;q^2)}{ h_k^P(c,d;z_-,z_+;q^2)} } P_k(y;c,d;q^2), & \text{if}\ \nu \neq \la,\\ \\
\displaystyle (\sgn(\nu))^k\sqrt{ \frac{|y| w(y;c,d;q^2)}{ h_k^Q(c,d;z_-,z_+;q^2)} } Q_k(y;c,d;z_-,z_+;q^2), & \text{if}\ \nu = \la,
\end{cases}
\]
with
\[
(c,d,z_-,z_+)=
\begin{cases}
(-\nu q^{1-\tau-\si}/x,\  -\nu x q^{1-\tau-\si}, \ -1, \ q^{2\si}), & \text{if} \ \nu \neq \la,\\
(xq^{1+\tau-\si}/\nu, \ q^{1+\tau-\si}/\nu x ,  -1, \  q^{2\si}),& \text{if}\ \nu = \la.
\end{cases}
\]
\end{enumerate}
\end{thm}
Note that $\la = -q^{2\tau}/\nu$ in case $\nu \neq \la$.
\begin{proof}
The case $\xi \neq \la$ is already proved, so we assume $\xi = \la$.
We first show that \eqref{eq:Sxyz1} is satisfied. The coefficients $\big(S_{x;y,\xi q^{2k}}^{\tau,\si}\big)_{\nu,\la}$ do not depend on $m$ and $n$, so we may choose $m$ and $n$ in a convenient way. We choose $n=m=0$, then for $z=\la q^{2k}$ with $k \in \N$ the explicit identity is, after canceling common terms,
\[
\begin{split}
\sum_{y \in -q^{2\Z} \cup q^{2\si + 2\Z}} & \big(S_{x;y,z}^{\tau,\si}\big)_{\nu,\la} \frac{|y|^{\frac12} (-yq^2, yq^{2-2\si};q^2)_\infty }{ (yq^{1+\tau-\si} x^{\pm 1}/\la;q^2)_\infty^{\frac12} } \\
=& \left(\frac{\te(-q^{2\si};q^2) (-\la q^{1-\tau+\si}x^{\pm 1},\la q^{1-\tau-\si} x^{\pm 1}, -\la q^{2}, \la q^{2-2\tau};q^2)_\infty }{ \te(-q^{2\tau+2};q^2)(-\la^2 q^{2-2\tau};q^2)_\infty } \right)^{\frac12}\\
& \times\left( \frac{|\nu|q^{2k-2\tau} (q^{2+2k},-\nu^2 q^{2-2\tau+2k}, \nu q^{2k-2\tau+2}, -\nu q^{2k+2};q^2)_\infty }{(q^2,-\nu q^{\si-\tau+1} x^{\pm 1}, \nu q^{1-\si-\tau} x^{\pm 1}, -\nu \la q^{2k-2\tau+2};q^2)_\infty } \right)^{\frac12}  \\
& \times   (\nu q^{1-\si-\tau}/x;q^2)_k x^k \rphis{2}{1}{q^{-2k}, -\nu x q^{\si -\tau +1}}{xq^{1+\si+\tau-2k}/\nu}{q^2, \frac{q^{1+\tau+\si}}{\nu x}}.
\end{split}
\]
For $k\in -\N_{\geq 1}$ the right hand side is equal to zero.
Note that $(-yq^2, yq^{2-2\si};q^2)_\infty=0$ for $y \not\in -q^{\N}\cup q^{2\si+2\N}$, so we can write the sum as a $q$-integral of the form $\int_{-1}^{q^{2\si}} f(y)\, d_{q^2}y$. Expressions for $\big(S_{x;y,z}^{\tau,\si}\big)_{\nu,\la}$ are now obtained from Proposition \ref{prop:integralPQ} with parameters
\[
(c,d,z_-,z_+)=
\begin{cases}
(-\nu q^{1-\tau-\si}/x,\  -\nu x q^{1-\tau-\si}, \ -1, \ q^{2\si}), & \nu \neq \la,\\
(xq^{1+\tau-\si}/\nu, \ q^{1+\tau-\si}/\nu x ,  -1, \  q^{2\si}),& \nu = \la.
\end{cases}
\]
The expressions obtained in this way can be simplified using the $\te$-product identity \eqref{eq:te-product} and
\begin{equation} \label{eq:handy identities}
\frac{q^{2\tau}}{|\la|} \te(-q^{2+2\tau};q^2) = \te(-\la^2q^{2-2\tau};q^2), \qquad
\frac{ (q^2,-\la^2 q^{2-2\tau};q^2)_k }{ (\la q^{2-2\tau}, -\la q^2;q^2)_k } =1.
\end{equation}

Next we show $S_{x;y,z}^{\tau,\si}$ is indeed the kernel in a unitary operator by verifying the orthogonality relations \eqref{eq:orthogonalityS}. First we assume $z_i  \in -q^{2\Z}$ for $i = 1,2$, then
\[
\begin{split}
\sum_{ y \in -q^{2\Z} \cup q^{2\si+2\Z} }& S_{x;y,z_1}^{\tau,\si} (S_{x;y,z_2}^{\tau,\si} )^* \\
& = \sum_{ y \in -q^{2\Z} \cup q^{2\si+2\Z} }
    \begin{pmatrix}
    \big(S_{x;y,z_1}^{\tau,\si}\big)_{-1,-1} & 0 \\ \big(S_{x;y,z_1}^{\tau,\si}\big)_{q^{2\tau},-1} & 0
    \end{pmatrix}
    \begin{pmatrix}
    \big(S_{x;y,z_2}^{\tau,\si}\big)_{-1,-1} & 0 \\ \big(S_{x;y,z_2}^{\tau,\si}\big)_{q^{2\tau},-1} & 0
    \end{pmatrix}^* \\
& = \sum_{ y \in -q^{2\Z} \cup q^{2\si+2\Z} }
\begin{pmatrix}
\big(S_{x;y,z_1}^{\tau,\si}\big)_{-1,-1}\cdot \big(S_{x;y,z_2}^{\tau,\si}\big)_{-1,-1}\ & \big(S_{x;y,z_1}^{\tau,\si}\big)_{-1,-1} \cdot\big(S_{x;y,z_2}^{\tau,\si}\big)_{q^{2\tau},-1} \\
\big(S_{x;y,z_1}^{\tau,\si}\big)_{q^{2\tau},-1}\cdot \big(S_{x;y,z_2}^{\tau,\si}\big)_{-1,-1}\ & \big(S_{x;y,z_1}^{\tau,\si} \big)_{q^{2\tau},-1}\cdot \big(S_{x;y,z_2}^{\tau,\si}\big)_{q^{2\tau},-1}
\end{pmatrix} \\
& = \de_{z_1z_2} \begin{pmatrix} 1 & 0 \\ 0 & 1 \end{pmatrix},
\end{split}
\]
where the last line follows from writing the matrix coefficients in terms of the functions $P_k$ and $Q_k$ and using the orthogonality relations \eqref{eq:orthogonalityPQ}. E.g., with $z_i=-q^{2k_i}$,
\[
\begin{split}
&\sum_{ y \in -q^{2\Z} \cup q^{2\si+2\Z} }\big(S_{x;y,z_1}^{\tau,\si}\big)_{-1,-1} \cdot \big(S_{x;y,z_2}^{\tau,\si}\big)_{q^{2\tau},-1} = \\& \quad K\int_{\infty(-1)}^{\infty(q^{2\si})} Q_{k_1}(y;c,d;-1,q^{2\si};q^2)P_{k_2}(y;c,d;q^2) w(y;c,d;q^2)\, d_{q^2}y = 0,
\end{split}
\]
where $(c,d) = (-xq^{1+\tau-\si},-q^{1+\tau-\si}/x)$ and $K$ is independent of $y$. The cases involving $z_i \in q^{2\tau+2\Z}$ are proved in the same way.

The dual orthogonality relations are also proved in this way, now using the dual orthogonality relations for $P_k$ and $Q_k$.
\end{proof}

As a result we obtain the following $q$-integral evaluation formulas involving Al-Salam--Carlitz II polynomials $P_n$ \eqref{eq:defPn} and functions $Q_n$ \eqref{eq:defQn}, continuous dual $q$-Hahn polynomials $p_n$ \eqref{eq:ContinuousDualqHahn}, Al-Salam--Chihara polynomials $q_n$ \eqref{eq:Al-Salam--Chihara}, big $q$-Laguerre polynomials $L_n$ \eqref{eq:bigqLaguerre}, and Al-Salam--Carlitz polynomials $U_n$ \eqref{eq:Al-Salam--Carlitz pol}.
\begin{thm} \label{thm:identityCC}
Let $x\in \T$, $s,t>0$, $\nu \in \{-1,t\}$, and $m,n \in \N$. For $k \in \N$,
\[
\begin{split}
\int_{-q^{-n}}^{sq^{-n}}& p_n(\mu_x;\tfrac1\nu\sqrt{qst},-\tfrac1\nu \sqrt{qt/s},-\nu y \sqrt{q/st};q)\,
U_m^{(-s)}(-yq^n;q)\\
& \times  P_k(y;-\nu x \sqrt{q/st}, -\tfrac{\nu}{x} \sqrt{q/st};q)\, \frac{ (-yq^{1+n},yq^{1+n}/s;q)_\infty }{(- y \nu x^{\pm 1} \sqrt{q/st};q)_\infty } d_q y \\
= & \,(1-q) (-1)^m \nu^{m+k} (q/st)^{k/2} (s/qt)^{m/2} q^{-n} q^{m(m-n)} q^{-\frac12k(k-1)} \\
& \times \frac{ (tq^{n-m+1}/\nu, - q^{n-m+1}/\nu;q)_m (q^{m-n+k+1},-\nu^2 q^{m-n+k+1}/t ;q)_\infty } {(-\nu x^{\pm 1} \sqrt{qs/t}, \nu x^{\pm 1} \sqrt{q/st};q)_\infty } \\
& \times q_{m-n+k}(\mu_x;-\nu \sqrt{qs/t}, \nu \sqrt{q/st};q)\, L_m(q^{k+1};tq^{n-m}/\nu, - q^{n-m}/\nu;q)
\end{split}
\]
and for $k \in \Z$
\[
\begin{split}
\int_{-q^{-n}}^{sq^{-n}}& p_n(\mu_x;-\nu\sqrt{qs/t},\nu \sqrt{q/st},\tfrac{y}{\nu} \sqrt{qt/s};q)\,
U_m^{(-s)}(-yq^n;q)\\
& \times  Q_k(y;\tfrac{x}{\nu}\sqrt{qt/s}, -\tfrac{1}{\nu x} \sqrt{qt/s};-1,s;;q)\, \frac{ (-yq^{1+n},yq^{1+n}/s;q)_\infty }{(- \tfrac{y x^{\pm 1}}{\nu} \sqrt{qt/s};q)_\infty } d_q y \\
= & \,(1-q) \nu^{-m-k} (t/qs)^{k/2} (st/q)^{m/2} q^{-n} q^{m(m-n)} q^{-\frac12k(k-1)} \\
& \times \frac{ (-\nu q^{n-m+1}, -\nu q^{n-m+1}/t;q)_m (q^{m-n+k+1},-\nu^2 q^{m-n+k+1}/t ;q)_\infty } {(-\nu^2 q^{k+1}/t;q)_\infty } \\
& \times q_{m-n+k}(\mu_x;-\nu \sqrt{qs/t}, \nu \sqrt{q/st};q)\, L_m(-\nu^2 q^{k+1}/t;-\nu q^{n-m},  \nu q^{n-m}/t;q).
\end{split}
\]
\end{thm}
\begin{proof}
This follows from writing \eqref{eq:Sxyz1} explicitly in terms of the corresponding special functions, substituting $(q^{2\si},q^{2\tau}) \mapsto (s,t)$ and replacing $q^2$ by $q$. The first $q$-integral identity corresponds to the case $\la \neq \nu$; the second identity corresponds to the case $\la = \nu$.
\end{proof}

We substitute $(x,\la,\nu) \mapsto (-q^x\sqrt{q/st},-1,-1)$ in the second identity of Theorem \ref{thm:identityCC}, and let $t \to 0$. This limit corresponds to $\rho_{\tau,\si} \to \rho_{\infty,\si}$ in $\mathcal A_q$. This gives the following identity involving big $q$-Laguerre polynomials $L_n$ \eqref{eq:bigqLaguerre}, Al-Salam--Carlitz polynomials $U_n$ \eqref{eq:Al-Salam--Carlitz pol}, Wall polynomials \eqref{eq:defWallpol}, Stieltjes-Wigert polynomials $S_n$ \eqref{eq:Stieltjes-Wigert} and related $q$-Bessel functions $M_n$ \eqref{eq:SW-qBessel}.
\begin{cor} \label{cor:identitySW-functions}
For $k,m,n,x \in \N$ and $s>0$,
\[
\begin{split}
&\sum_{y\in \N} L_n(q^{x+y+1};-sq^y,q^y;q)\, U_m^{(-s)}(-sq^{y+n};q)\, S_{x+y}(sq^{k-2x};q)\\
& \qquad \times (q^{y+1};q)_x(-sq^{y+1};q)_\infty (sq^y)^{-n} sq^{y+x}\\
+& \sum_{y \in \N} L_n(-q^{x+y+1}/s;q^y, -q^y/s;q)\, U_m^{(-s)}(q^{y+n};q) \, M_{x+y}^{(s)}(sq^{k-2x};q)\\
& \qquad \times (-q^{y+1}/s;q)_x (q^{y+1};q)_\infty q^{-ny} q^{y+x} \\
= &\ s^{x-k} q^{-x(x-1)} q^{-\frac12n(n+1)} q^{-\frac12k(k-1)} q^{n(k+m)+k(x-m)} \te(-s;q)\\
& \times (q^{n-m+1};q)_m (q^{m-n+k+1};q)_\infty\, U_{m-n+k}^{(-s)}(q^x;q)\, p_m(q^{m-n+k};q^{n-m};q).
\end{split}
\]
\end{cor}
\begin{remark}\*
\begin{enumerate}[(i)]
\item Define the $\mathcal T(\rho_{\infty,\si})$-eigenvectors $U^{\infty,\si}_{x,p} \in \ell^2(\N)\tensor \ell^2(\N)$ by
\[
U_{x,p}^{\infty,\si} = \sum_{n \in \N} m_{x,n}^{\infty,\si} V_{n,p+n}, \qquad p \in \Z,
\]
where $V_{n,p}$ is the eigenvector of $\mathcal T(\ga\ga^*)$ defined by \eqref{eq:eigenvector V infty infty}.
$U^{\infty,\si}_{x,p}$ is an eigenvector of $\mathcal T(\rho_{\infty,\si})$ for eigenvalue $x \in -q^{2\N}\cup q^{2\si+2\N}$.
Corollary \ref{cor:identitySW-functions} shows that the coupling coefficients between $\mathcal T(\rho_{\infty,\si})$-eigenvectors $V^{\infty,\si}_{x,y}$ and $U_{x,p}^{\infty,\si}$ can be expressed in terms of Stieltjes-Wigert polynomials $S_k$ and related $q$-Bessel functions $M_k$.
\item In the $(\infty,\infty)$-case there is no eigenvector similar to $U^{\tau,\si}$ or $U^{\infty,\si}$, so the limit $s \to 0$ in Corollary \ref{cor:identitySW-functions} will not give any interesting identity.
\end{enumerate}
\end{remark}

\bigskip

Similarly as the generalized eigenvector \eqref{eq:eigenvectorUtausi}, we also have the generalized eigenvector
\[
\widetilde U_{x,y}^{\tau,\si;\nu}=\sum_{n \in \N}  m_{x,\nu q^{2n}}^{\si,\tau} V_{\nu q^{2n},yq^{-2n}}^{\infty,\si},\qquad \nu =-1,q^{2\si},
\]
of $\mathcal T(\rho_{\tau,\si})$ for eigenvalue $\mu_x$. Here $y \in -q^{2\Z} \cup q^{2\si+2\Z}$. Essentially this is the generalized eigenvector $U_{x,y}^{\tau,\si;\nu}$ with $\tau \leftrightarrow \si$. We define, for $y \in -q^{2\Z} \cup q^{2\si+2\Z}$, $z \in -q^{2\Z} \cup q^{2\tau+2\Z}$, the $2\times2$-matrix valued coupling coefficients $T_{x;y,z}^{\tau,\si}$ by
\[
\begin{pmatrix}
\widetilde U^{\tau,\si;-1}_{x,y} \\ \widetilde U^{\tau,\si;q^{2\si}}_{x,y}
\end{pmatrix}
= \sum_{z \in -q^{2\Z} \cup q^{2\tau+2\Z} } T^{\tau,\si}_{x;y,z}
\begin{pmatrix}
U^{\tau,\si;-1}_{x,z} \\ U^{\tau,\si;q^{2\tau}}_{x,z}
\end{pmatrix}.
\]
These coupling coefficients can be expressed in terms of the the coupling coefficients $S$ as follows.

We start by expanding $\widetilde U^{\tau,\si}_{x,y}$ in terms of $\widetilde V^{\tau,\si}_{x,z}$, using \eqref{eq:U=sum SV} with $\tau \leftrightarrow \si$;
\[
\begin{split}
\begin{pmatrix}
\widetilde U^{\tau,\si;-1}_{x,y} \\ \widetilde U^{\tau,\si;q^{2\si}}_{x,y}
\end{pmatrix}
&= \sum_{k \in \Z} \widetilde S_{x;-q^{2k},y}^{\tau,\si}
\begin{pmatrix}
\widetilde V_{x,-q^{2k}}^{\tau,\si;-1} \\ \widetilde V_{x,-q^{2k}}^{\tau,\si;q^{2\si}}
\end{pmatrix}
+ \widetilde S_{x;q^{2\tau+2k},y}^{\tau,\si}
\begin{pmatrix}
\widetilde V_{x,q^{2\tau+2k}}^{\tau,\si;-1} \\ \widetilde V_{x,q^{2\tau+2k}}^{\tau,\si;q^{2\si}}
\end{pmatrix} \\
& = \sum_{k \in \Z} \widetilde S_{x;-q^{2k},y}^{\tau,\si}
\begin{pmatrix}
V_{x,-q^{-2k}}^{\tau,\si;-1} \\  V_{x,q^{2\si-2k}}^{\tau,\si;-1}
\end{pmatrix}
+ \widetilde S_{x;q^{2\tau+2k},y}^{\tau,\si}
\begin{pmatrix}
V_{x,-q^{-2k}}^{\tau,\si;q^{2\tau}} \\ V_{x,q^{2\tau-2k}}^{\tau,\si;q^{2\tau}}.
\end{pmatrix}.
\end{split}
\]
For the second identity we used \eqref{eq:tildeV<->V} to write $\widetilde V_{x,v}^{\tau,\si}$ in terms of $V_{x,v}^{\tau,\si}$.
Then we transform $V_{x,v}^{\tau,\si}$ to $U_{x,z}^{\tau,\si}$ using the inverse of \eqref{eq:U=sum SV},
\[
\begin{pmatrix}
V^{\tau,\si;\nu}_{x,-q^{-2k}} \\ V^{\tau,\si;\nu}_{x,q^{2\si-2k} }
\end{pmatrix}
= \sum_{z \in -q^{2\Z} \cup q^{2\tau+2\Z} }
\begin{pmatrix}
\big(S_{x;-q^{-2k},z}^{\tau,\si}\big)_{-1,\nu} & \big(S_{x;-q^{-2k},z}^{\tau,\si}\big)_{q^{2\tau},\nu} \\
\big(S_{x;q^{2\si-2k},z}^{\tau,\si}\big)_{-1,\nu} & \big(S_{x;q^{2\si-2k},z}^{\tau,\si}\big)_{q^{2\tau},\nu}
\end{pmatrix}
\begin{pmatrix}
U^{\tau,\si;-1}_{x,z} \\ U^{\tau,\si;q^{2\tau}}_{x,z}
\end{pmatrix}.
\]
Combining then gives the following result:
\[
T^{\tau,\si}_{x;y,z} = \sum_{v \in -q^{2\Z}\cup q^{2\tau+2\Z}} \widetilde S_{x;v,y}^{\tau,\si} \hat S_{x;v,z}^{\tau,\si},
\]
where $\widetilde S_{x;v,y}^{\tau,\si} = S_{x;v,y}^{\si,\tau}$, and
\[
\hat S_{x;v,z}^{\tau,\si} =
\begin{pmatrix}
\big(S_{x;-q^{-2k},z}^{\tau,\si}\big)_{-1,\nu} & \big(S_{x;-q^{-2k},z}^{\tau,\si}\big)_{q^{2\tau},\nu} \\
\big(S_{x;q^{2\si-2k},z}^{\tau,\si}\big)_{-1,\nu} & \big(S_{x;q^{2\si-2k},z}^{\tau,\si}\big)_{q^{2\tau},\nu}
\end{pmatrix},
\qquad v = \nu q^{2k}.
\]
The coupling coefficients $T^{\tau,\si}_{x;y,z}$ are self-dual, i.e.~$(T^{\tau,\si}_{x;y,z})^* = T^{\si,\tau}_{x;z,y}$, and they satisfy the orthogonality relations
\[
\sum_{z \in -q^{2\Z} \cup q^{2\tau+2\Z} } T^{\tau,\si}_{x;y_1,z} \left(T^{\tau,\si}_{x;y_2,z}\right)^* = \de_{y_1y_2} I.
\]

\section{Coupling coefficients for three-fold tensor products} \label{sec:threefold}
In this section we consider a three-fold tensor product representation of $\mathcal A_q$ and we compute the coupling coefficients between two different generalized eigenvectors of $\rho_{\tau,\si}$. The coupling coefficients can be considered as $6j$-symbols or Racah coefficients. We start by introducing a class of functions that we need later on.

\subsection{$q$-Meixner functions}
The $q$-Meixner function is introduced in \cite{GroenK}, and is defined by
\begin{equation} \label{def:qMeixner}
\phi_\ga(x)=\phi_\ga(x;a,b;q) = \rphis{2}{2}{-1/x,-1/\ga}{a,b}{q,ab\ga x}.
\end{equation}
It is an entire function in $x$ and in $\ga$, and it is self-dual, i.e., $\phi_\ga(x) = \phi_x(\ga)$. Note that, for $n \in \N$, $\phi_{-q^n}(x)$ is a polynomial in $x$ of degree $n$; the $q$-Meixner polynomial. For parameters $a,b,t$ satisfying $a=\overline{b}$, $a \in \C\setminus\R$, $t>0$, the $q$-Meixner functions satisfy the orthogonality relations
\begin{equation} \label{eq:orthogonality qMeixner}
\int_{-1}^{\infty(t)} \phi_{\ga_1}(x)\phi_{\ga_2}(x) w(x)\, d_q x = \de_{\ga_1,\ga_2} (1-q) \frac{K_t \, K_{q/abt}}{|\ga_1| w(\ga_1)}, \qquad \ga_1,\ga_2 \in -q^\N \cup q^\Z/abt,
\end{equation}
where
\begin{gather*}
w(x) = w(x;a,b;q) = \frac{ (-qx;q)_\infty }{(-ax,-bx;q)_\infty}, \\
K_t = K_t(a,b;q) = \frac{ (q;q)_\infty \te(-t;q)}{(a,b,-at,-bt;q)_\infty}.
\end{gather*}
Furthermore, $\{\phi_\ga\mid \ga \in -q^\N \cup q^\Z/abt\}$ is an orthogonal basis for the corresponding $L^2$-space. Note that the dual orthogonality relations are equivalent to \eqref{eq:orthogonality qMeixner}.

The following $q$-integral involving $q$-Meixner functions is useful; the proof is given in Appendix \ref{app:integral qMeixner}.
\begin{prop} \label{prop:integral qMeixner}
For $m \in \N$,
\[
\begin{split}
\int_{-1}^{t} &\phi_\ga(xq^m;aq^{-m}, bq^{-m};q) (xq/t;q)_\infty w(x;a,b;q)\, d_q x   = \\ &(1-q) K_{tq^m}(aq^{-m},bq^{-m};q)\, (abt\ga q^{-m};q)_\infty (-qt;q)_m q^{-m} \rphis{3}{2}{ q^{-m},-bt,-at}{abt\ga q^{-m}, -qt}{q,q}.
\end{split}
\]
\end{prop}

\textbf{Limit cases.} The $q$-Meixner functions can be considered as generalizations of certain $q$-Bessel functions.
We set $t= c/a$, $a= q^{\al+1}$, $b=\be q^n$ with $\al>-1$, $c>0$ and $n \in \N$,  then the limit $n\to \infty$ of the $q$-Meixner function with $\ga \in q^\Z/abt$ gives
\begin{equation} \label{eq:bigqBessel}
\lim_{n \to \infty} \phi_{q^{k-n+1}/\be c}(x;q^{\al+1},\be q^n;q) = \rphis{1}{1}{-1/x}{q^{\al+1}}{q, \frac{x q^{k+\al+2}}{c}} = \mathcal J_{\al,k}^c(-x;q),
\end{equation}
where $\mathcal J_{\al,k}^c(-x;q)$ denotes a big $q$-Bessel function \cite[(6.1)]{CKK}. For $x \in -q^\N$ this is a $q$-Laguerre polynomial in $q^k$. Furthermore, for $\ga \in -q^\N$ we obtain
\[
\lim_{n \to \infty} \phi_{-q^k}(x;q^{\al+1},\be q^n;q) = 0.
\]
Taking the limit in the orthonality relations \eqref{eq:orthogonality qMeixner} for the $q$-Meixner functions we find formally
\begin{gather*}
\begin{split}
\int_{-1}^{\infty(cq^{-1-\al})} &  \mathcal J_{\al,k}^c(-x;q) \mathcal J_{\al,l}^c(-x;q) \frac{(-qx;q)_\infty}{(-xq^{\al+1};q)_\infty}\,d_qx = \\
& \de_{k,l} (1-q) \frac{ (q;q)^2_\infty \te(-q^{\al+2}/c;q) }{ (q^{\al+1};q)_\infty^2 \te(-q/c;q) } q^{-k(\al+1)} (-q^{k+1}/c;q)_\infty, \qquad k,l \in \Z,
\end{split} \\
\begin{split}
\sum_{k \in \Z} \mathcal J_{\al,k}^c(-x;q) &\mathcal J_{\al,k}^c(-y;q) \frac{q^{k(\al+1)}}{(-q^{k+1}/c;q)_\infty} =\\ & \de_{x,y} \frac{ (q;q)^2_\infty \te(-q^{\al+2}/c;q) }{ (q^{\al+1};q)_\infty^2 \te(-q/c;q) } \frac{(-xq^{\al+1};q)_\infty}{|x|(-qx;q)_\infty}, \qquad x,y \in -q^\N \cup cq^{\Z-1-\al},
\end{split}
\end{gather*}
where we used the $\te$-product identity \eqref{eq:te-product} for the right hand side of \eqref{eq:orthogonality qMeixner}. These orthogonality relations are proved by Ciccoli, Koelink and Koornwinder in \cite{CKK}, so there is no need to make the limit transitions rigorous. Note that the self-duality property of the $q$-Meixner functions is lost in the limit.\\

We can take one more limit. We set $c=q^{-m}$, $m \in \Z$ and let $x \in cq^{\Z-\al-1}$, then
\begin{equation} \label{eq:HahnExtonqBessel}
\lim_{m \to \infty} \mathcal J_{\al,k}^{q^{-m}}(-q^{n-m-\al-1};q) = \rphis{1}{1}{0}{q^{\al+1}}{q,q^{n+k+1}} = J_\al(q^{n+k};q),
\end{equation}
where $J_\al$ is the Hahn-Exton $q$-Bessel function, or Jackson's third $q$-Bessel function. For $x \in -q^\N$ we have
\[
\lim_{m \to \infty} \mathcal J_{\al,k}^{q^{-m}}(-q^{n};q)=0.
\]
The limit of the orthogonality relations becomes
\[
\sum_{k \in \Z} J_\al(q^{k+m};q^2)J_\al(q^{k+n};q^2) q^{\al(k+1)} = \de_{m,n} q^{-n(\al+1)}  \frac{ (q;q)^2_\infty}{ (q^{\al+1};q)_\infty^2  }, \qquad m,n \in \Z,
\]
which are the $q$-Hankel orthogonality relations proved by Koorwinder and Swarttouw in \cite{KooSw}. Note that these $q$-Bessel functions are again self-dual.

\begin{remark}
The Al-Salam--Carlitz II functions from \S\ref{ssec:ACSfunctions} and their orthogonality relations can also formally be obtained from the $q$-Meixner functions. We will not need this in this paper.
\end{remark}

\subsection{Coupling coefficients}
We consider the threefold tensor product representations of $\mathcal A_q$ on $\ell^2(\N)^{\tensor 3}$ given by
\[
\mathcal T^{(3)} = (\pi_0 \tensor \pi_0 \tensor \pi_0) (1\tensor \De)\De.
\]
We first determine eigenvectors of $\mathcal T^{(3)}(\rho_{\tau,\si})$ in the same way as in \S\ref{ssec:CGCtausi}. From \eqref{eq:De(rhotausi)} we find
\[
\begin{split}
(1\tensor \De)&\De(\rho_{\tau,\si})=\\
& \hf q^{-\tau-\si-1} \Big( q^{-1}\, \al_{\tau+1,\infty} \ga_{\tau,\infty} \tensor \De(\al_{\infty, \si+1} \be_{\infty,\si})
+ q\, \be_{\tau+1,\infty} \de_{\tau, \infty} \tensor \De(\ga_{\infty, \si+1} \de_{\infty,\si})\\
& + q^2(1+q^2)\, \rho_{\tau,\infty} \tensor \De(\rho_{\infty, \si}) + q^2(1-q^{2\si})\, \rho_{\tau,\infty} \tensor \De(1) + q^2(1-q^{2\tau})\, 1 \tensor \De(\rho_{\infty, \si}) \Big).
\end{split}
\]
We now let $\mathcal T^{(3)}(\rho_{\tau,\si})$ act on the $\ell^2(\N)^{\tensor 3}$-basis  $\{v_x^{\tau,\infty} \tensor V_{y,z}^{\infty,\si} \mid x \in -q^{2\N}\cup q^{2\tau+2\N}, y \in -q^{2\N}\cup q^{2\si+2\N}, z \in -q^{2\Z} \cup q^{2\si+2\Z}\}$, see Propositions \ref{prop:eigenvector rho tau infty}, \ref{prop:actions on v tau infty} and \eqref{eq:T action V infty si}. This gives
\[
\begin{split}
\mathcal T^{(3)}(\rho_{\tau,\si})\, &v_{x}^{\tau,\infty} \tensor V_{y,z}^{\infty,\si} = \\
&\hf \Big[xy q^{1-\tau -\si}(1+ q^2) + y q^{1-\tau -\si} (1- q^{2\si})  + y q^{1-\tau -\si}(1-q^{2\tau})\Big]\,  v_{x}^{\tau,\infty} \tensor V_{y,z}^{\infty,\si}\\
+& \hf \sqrt{ (1+ x ) (1- x q^{-2\tau} ) ( 1+ y) (1-y q^{-2\si} ) } \, v_{x/q^2}^{\tau,\infty} \tensor V_{y/q^2, z/q^2}^{\infty,\si} \\
+& \hf  \sqrt{ (1+ x q^{2} ) (1- x q^{2-2\tau} ) ( 1+ y q^{2} ) (1-y q^{2-2\si} ) } \, v_{x q^{2}}^{\tau,\infty} \tensor V_{y q^{2},zq^2 }^{\infty,\si},
\end{split}
\]
This can be matched to the three-term recurrence relation for certain continuous dual $q^2$-Hahn polynomials, and then similar as in \S\ref{ssec:CGCtausi} we find the generalized eigenvectors
\[
F_{u,y,z}^{\tau,\si;\la} =\sum_{n \in \N } c_{u,\la q^{2n},y}^{\tau,\si}\, v_{\la q^{2n}}^{\tau,\infty} \tensor V_{yq^{2n},zq^{2n}}^{\infty,\si}, \qquad \la = -1, q^{2\tau},
\]
for eigenvalue $\mu_u$.

Next we determine another generalized eigenvector of $\mathcal T^{(3)}(\rho_{\tau,\si})$ using the cocommutativity of $\De$, i.e.,~$(1\tensor \De)\De = (\De \tensor 1)\De$. From \eqref{eq:De(rhotausi)} we find
\[
\begin{split}
(\De\tensor 1)&\De(\rho_{\tau,\si})=\\
& \hf q^{-\tau-\si-1} \Big( q^{-1}\, \De(\al_{\tau+1,\infty} \ga_{\tau,\infty}) \tensor\al_{\infty, \si+1} \be_{\infty,\si}
+ q\, \De(\be_{\tau+1,\infty} \de_{\tau, \infty}) \tensor \ga_{\infty, \si+1} \de_{\infty,\si}\\
& + q^2(1+q^2)\, \De(\rho_{\tau,\infty}) \tensor \rho_{\infty, \si} + q^2(1-q^{2\si})\, \De(\rho_{\tau,\infty}) \tensor1 + q^2(1-q^{2\tau})\, \De(1) \tensor\rho_{\infty, \si} \Big).
\end{split}
\]
Now we let $\mathcal T^{(3)}(\rho_{\tau,\si})$ act on the $\ell^2(\N)^{\tensor 3}$-basis vectors $V^{\tau,\infty}_{x,y} \tensor v^{\infty,\si}_z$,
\[
\begin{split}
\mathcal T^{(3)}(\rho_{\tau,\si})\, &V_{x,y}^{\tau,\infty} \tensor v_{z}^{\infty,\si} = \\
&\hf \Big[xz q^{1-\tau -\si}(1+ q^2) + z q^{1-\tau -\si} (1- q^{2\si})  + z q^{1-\tau -\si}(1-q^{2\tau})\Big]\,  V_{x,y}^{\tau,\infty} \tensor v_{z}^{\infty,\si}\\
+& \hf  \sqrt{ (1+ x ) (1- x q^{-2\tau} ) ( 1+ z) (1-z q^{-2\si} ) } \, V_{x/q^2,yq^2}^{\tau,\infty} \tensor v_{z/q^2}^{\infty,\si} \\
+& \hf  \sqrt{ (1+ x q^{2} ) (1- x q^{2-2\tau} ) ( 1+ z q^{2} ) (1-z q^{2-2\si} ) } \, V_{x q^{2},y/q^2}^{\tau,\infty} \tensor v_{z q^{2} }^{\infty,\si},
\end{split}
\]
and then we obtain the generalized eigenvector
\[
G_{u,x,y}^{\tau,\si;\xi}=\sum_{k \in \N } c_{u,\xi q^{2k},x}^{\si,\tau}\, V_{x q^{2k},yq^{-2k}}^{\tau,\infty} \tensor v_{\xi q^{2k}}^{\infty,\si}, \qquad \xi = -1, q^{2\si},
\]
for eigenvalue $\mu_u$ (compare with Remark \ref{rem:si<->tau}).\\

We define $R_{u;x,y;z,v}^{\tau,\si}$ as the $2\times2$-matrix-valued coupling coefficients between the generalized eigenvectors;
\begin{equation} \label{eq:defR}
\begin{gathered}
\begin{pmatrix}
F_{u,x,y}^{\tau,\si;-1} \\ F_{u,x,y}^{\tau,\si;q^{2\tau}}
\end{pmatrix}
= \sum_{z,v} R_{u;x,y;z,v}^{\tau,\si}
\begin{pmatrix}
G_{u,z,v}^{\tau,\si;-1} \\ G_{u,z,v}^{\tau,\si;q^{2\si}}
\end{pmatrix},
\\
R_{u;x,y;z,v}^{\tau,\si} = \left(\big(R_{u;x,y;z,v}^{\tau,\si}\big)_{\la,\xi} \right)_{\la \in \{-1, q^{2\tau}\},\, \xi \in \{-1, q^{2\si}\} }.
\end{gathered}
\end{equation}
The coupling coefficients satisfy the orthogonality relations
\begin{equation} \label{eq:orthogonalityR}
\begin{split}
\sum_{x,y\in -q^{2\Z} \cup q^{2\si+2\Z}} (R_{u;x,y;z_1,v_1}^{\tau,\si})^* R_{u;x,y;z_1,v_2}^{\tau,\si} = \de_{z_1,z_2} \de_{v_1,v_2} I,\\
\sum_{z,v\in -q^{2\Z} \cup q^{2\tau+2\Z}} R_{u;x_1,y_1;z,v}^{\tau,\si} (R_{u;x_2,y_2;z,v}^{\tau,\si})^*= \de_{x_1,x_2}\de_{y_1,y_2} I,
\end{split}
\end{equation}
where $I$ is the $2\times2$ identity matrix, and $A^*$ denotes the complex transpose of $A$.
Furthermore, since $(x,y,z,v,\si,\tau) \to (z,v,x,y,\tau,\si)$ interchanges $F_{u,x,y}^{\tau,\si}$ with $G_{u,z,v}^{\tau,\si}$, we have the symmetry property
\begin{equation} \label{eq:symmetryR}
\big(R^{\tau,\si}_{u;x,y;z,v}\big)_{\la,\xi} = \big(R^{\si,\tau}_{u;z,v;x,y}\big)_{\xi,\la}.
\end{equation}
Because of this symmetry property the two orthogonality relations \eqref{eq:orthogonalityR} are equivalent.\\

We use \eqref{eq:defR} to determine an explicit expression for $\big(R^{\tau,\si;}_{u;x,y;z,v}\big)_{\la,\xi}$. First a preliminary result.
\begin{lemma}
The coupling coefficients satisfy
\begin{equation} \label{eq:R=deltadeltaR}
\big(R^{\tau,\si}_{u;x,y;z,v}\big)_{\la,\xi} = \de_{y,\xi q^{2r}}\, \de_{v, \la q^{2r}}\, R^{\tau,\si;\la,\xi}_{u,r;x,z}, \qquad r \in\Z,
\end{equation}
where $R^{\tau,\si;\la,\xi}_{u,r;x,z}$ is determined by
\begin{equation} \label{eq:sum Rcc=cc}
\sum_{z\in -q^{2\Z} \cup q^{2\tau + 2\Z}} R^{\tau,\si;\la,\xi}_{u,r;x,z}\, c_{u,\xi q^{2k},z}^{\si,\tau} c^{\tau,\infty}_{zq^{2k},\la q^{2n-2m},m} = c_{u,\la q^{2n},x}^{\tau,\si} c^{\infty,\si}_{xq^{2n},m,\xi q^{2k-2m}}, \qquad r = n-m+k,
\end{equation}
for $k,m,n \in \N$.
\end{lemma}
\begin{proof}
We have
\[
\langle F_{u,x,y}^{\tau,\si;\nu}, v_{\la q^{2n}}^{\tau,\infty}\tensor e_m \tensor v_{\xi q^{2k}}^{\infty, \si} \rangle = \de_{yq^{2n+2m},\xi q^{2k}} \de_{\nu,\la}\, c_{u,\la q^{2n},x}^{\tau,\si} c_{xq^{2n},m, yq^{2n}}^{\infty,\si},
\]
and
\[
\langle G_{u,z,v}^{\tau,\si;\nu}, v_{\la q^{2n}}^{\tau,\infty}\tensor e_m \tensor v_{\xi q^{2k}}^{\infty, \si} \rangle = \de_{\la q^{2n}, vq^{2m-2k}} \de_{\nu,\xi}\, c_{u,\xi q^{2k},z}^{\si,\tau} c_{z q^{2k}, vq^{-2k}, m}^{\tau,\infty}.
\]
Now taking the inner product of \eqref{eq:defR} with basis vectors $v_{\la q^{2n}}^{\tau,\infty}\tensor e_m \tensor v_{\xi q^{2k}}^{\infty, \si}$ gives the result.
\end{proof}

Note that by \eqref{eq:R=deltadeltaR}, for fixed $y,v,\xi,\la$, only one of the four matrix coefficients of $R^{\tau,\si}_{u;x,y;z,v}$ is nonzero.
We are now ready to show that the coefficients $R^{\tau,\si;\la,\xi}_{u,r;x,z}$ are orthonormal Meixner functions.
\begin{thm}
Let $u \in \T$, $\la \in \{-1,q^{2\tau}\}$, $\xi \in \{-1,q^{2\si}\}$, $r \in \Z$, $x \in -q^{2\Z}\cup q^{2\si+2\Z}$ and $z \in -q^{2\Z} \cup q^{2\tau+2\Z}$ and define $x'= \xi x q^{2r-2\si}$ and $z'= \la z q^{2r-2\tau}$. Then for $x' \in -q^{2\N}\cup q^{2\si+2\Z}$ and $z' \in -q^{2\N}\cup q^{2\tau+2\Z}$ we have
\[
R^{\tau,\si;\la,\xi}_{u,r;x,z} = \sgn(\la \xi)^{r}
 \sqrt{\frac{|x' z'|w(z';a,b;q^2) w(x' ;a,b;q^2)}{ K_{t}(a,b;q^2) K_{q^2/abt}(a,b;q^2)}}\, \phi_{x'}(z';a,b;q^2),
\]
with
\[
(a,b,t) =
(-q^{1+\si+\tau-2r}u/\la\xi, -q^{1+\si+\tau-2r}/u \la\xi, \la^2 q^{2r-2\tau}).
\]
For $x' \in -q^{-2-2\N}$ or $z' \in -q^{-2-2\N}$, we have $R^{\tau,\si;\la,\xi}_{u,r;x,z}=0$.
\end{thm}
\begin{proof}
We need to verify that \eqref{eq:sum Rcc=cc} is satisfied, and we need to verify the orthogonality relations \eqref{eq:orthogonalityR} for $R^{\tau,\si}_{u;x,y;z,v}$. The orthogonality relations follow directly from the orthogonality relations \eqref{eq:orthogonality qMeixner} for the $q$-Meixner functions.

Since the coupling coefficients are independent of $k$ and $m$, it is enough to show that \eqref{eq:sum Rcc=cc} holds for $k=m=0$ (so $r=n$), i.e.
\begin{equation} \label{eq:hulpeqR}
\begin{split}
&\left( \frac{q^{-2\tau}(q^2,-\xi^2 q^{2-2\si},-\la q^{2+2n},\la q^{2+2n-2\tau};q^2)_\infty } { (-\xi q^{1-\si+\tau} u^{\pm 1}, \xi q^{1-\tau-\si} u^{\pm 1};q^2)_\infty  \te(-q^{-2\tau};q^2)} \right)^{\frac12} \\
& \times \sum_{z\in -q^{2\Z}\cup q^{2\tau + 2\Z}} R^{\tau,\si;\la,\xi}_{u,n;x,z}
\frac{|z|^{\frac12}(zq^{2-2\tau}, -zq^2;q^2)_\infty}{(zq^{1+\si-\tau}u^{\pm 1}/\xi,-\la z q^{2+2n-2\tau};q^2)_\infty^{1/2}}\\
& \quad = \frac{|x|^{\frac12}(xq^{2+2n-2\si}, -xq^2;q^2)_\infty}{( xq^{1+\tau-\si}u^{\pm 1}/\la,-\xi x q^{2+2n-2\si};q^2)_\infty^{1/2}} (-\la q^{\si-\tau})^{-n}\\
& \qquad \times \left( \frac{q^{-2\si}(q^{2n+2},-\la^2 q^{2+2n-2\tau},-\xi q^2,\xi q^{2-2\si};q^2)_\infty } { (-\la q^{1-\tau+\si} u^{\pm 1}, \la q^{1-\tau-\si} u^{\pm 1};q^2)_\infty \te(-q^{-2\si};q^2)} \right)^{\frac12}\\
& \qquad \times (-\la^2 q^{2-2\tau};q^2)_n  \rphis{3}{2}{q^{-2n}, - \la q^{1-\tau+\si} u , -\la q^{1-\tau+\si}/u}{-\la^2 q^{2-2\tau}, - x q^2}{q^2,q^2}.
\end{split}
\end{equation}
We will use Proposition \ref{prop:integral qMeixner}. First note that the sum over $z$ in \eqref{eq:hulpeqR} can be written as a $q$-integral of the form $\int_{-1}^{q^{2\tau}} f(z) \, d_{q^2}z$. The summand can be expressed in terms of the $q$-Meixner weight function $w$ \eqref{eq:orthogonality qMeixner} using
\[
\frac{(zq^{2-2\tau}, -zq^2;q^2)_\infty}{(zq^{1+\si-\tau}u^{\pm 1}/\xi,-\la z q^{2+2n-2\tau};q^2)_\infty^{1/2}} = \frac{w(\la z q^{-2\tau};\tilde a,\tilde b;q^2)}{w(\la z q^{2n-2\tau};\tilde a q^{-2n},\tilde b q^{-2n};q^2)^{1/2}} (zq^2/\la;q^2)_\infty,
\]
where $(\tilde a,\tilde b) =(-q^{1+\si+\tau}u/\la\xi, -q^{1+\si+\tau}/u \la\xi )$. Note that the right hand side of \eqref{eq:hulpeqR} contains a similar expression. Now we can verify, using identities like \eqref{eq:handy identities} that for $\xi=-1$ identity \eqref{eq:hulpeqR} can be written as
\[
\begin{split}
\frac{1}{1-q}\int_{-1}^{q^{2\tau}} &\frac{R^{\tau,\si;\la,\xi}_{u,r;x,z}\,(zq^2/\la;q^2)_\infty }{w(\la z q^{2n-2\tau};\tilde a q^{-2n},\tilde b q^{-2n};q^2)^{1/2}}\,w(\la z q^{-2\tau};\tilde a,\tilde b;q^2)|z|^{-\frac12}\, d_{q^2}z = \\& \sgn(\la \xi)^{n}|\xi x/\la|^\frac12 q^{\tau - \si} \left(\frac{w(\xi x q^{2n-2\si};\tilde a q^{-2n},\tilde b q^{-2n};q^2) K_{\tilde tq^{2n}}(\tilde a q^{-2n},\tilde b q^{-2n};q^2)}{ K_{q^{2+2n}/\tilde a\tilde b \tilde t} (\tilde a q^{-2n},\tilde b q^{-2n};q^2) } \right)^\frac12 \\
&\times  (xq^{2}/\xi;q^2)_\infty  (-q^2\tilde t;q^2)_n \rphis{3}{2}{q^{-2n}, - \tilde a \tilde t, -\tilde b \tilde t}{-q^2 \tilde t, x q^2/\xi}{q^2,q^2},
\end{split}
\]
with $\tilde t = \la^2 q^{-2\tau}$. For $\xi = q^{2\si}$ this is also true, but we need to write the $_3\varphi_2$-function in \eqref{eq:hulpeqR} as
\begin{multline*}
\qquad \rphis{3}{2}{q^{-2n}, -\la q^{1-\tau+\si} u , -\la q^{1-\tau+\si}/u}{-\la^2 q^{2-2\tau}, - x q^2}{q^2,q^2}
=\\
(-1)^n q^{2n\si} \frac{ (x q^{2-2\si};q^2)_n}{ (-xq^2;q^2)_n}\rphis{3}{2}{q^{-2n},  \la q^{1-\tau-\si} u , \la q^{1-\tau-\si}/u}{-\la^2 q^{2-2\tau}, x q^{2-2\si}}{q^2,q^2}, \qquad
\end{multline*}
which follows from the $_3\varphi_2$-transformation \cite[(III.11)]{GR}. For $\la = q^{2\tau}$ it now is a straightforward calculation using Proposition \ref{prop:integral qMeixner} to show that \eqref{eq:hulpeqR}, hence also \eqref{eq:sum Rcc=cc}, is satisfied with the expression for $R^{\tau,\si;\la,\xi}_{u,r;x,z}$ given in the theorem. For $\la = -1$ the calculation is the same after application of the substitution rule for $q$-integrals,
\[
\int_{-1}^{q^{2\tau}} f(-zq^{-2\tau})\,d_{q^2}z = q^{2\tau} \int_{-1}^{q^{-2\tau}} f(z)\, d_{q^2} z. \qedhere
\]
\end{proof}

\begin{remark}
Note that the symmetry property \eqref{eq:symmetryR} for the coupling coefficients corresponds to the self-duality property of the $q$-Meixner functions.
\end{remark}

Formula \eqref{eq:sum Rcc=cc} now gives a $q$-integral identity involving continuous dual $q$-Hahn polynomials $p_n$ \eqref{eq:ContinuousDualqHahn}, big $q$-Laguerre polynomials $L_n$ \eqref{eq:bigqLaguerre} and $q$-Meixner functions $\phi_\ga$ \eqref{def:qMeixner}.
\begin{thm} \label{thm:sum 2phi2 3phi2 3phi2}
Let $u \in \T$, $k,m,n \in \N$, $s,t>0$, $x \in -q^{\N-n} \cup sq^{\N-n}$, $\xi \in \{-1,s\}$, $\la \in \{-1,t\}$, and set $r=k-m+n$, then
\[
\begin{split}
\int_{-q^{-k}}^{tq^{-k}}&  p_k(\mu_u;-\xi \sqrt{qt/s}, \xi \sqrt{q/st}, \tfrac{z}{\xi} \sqrt{qs/t};q) \,
L_m(-\la z q^{1+r}/t;-\la q^{n-m}, \la q^{n-m}/t;q) \\
&\times \phi_{\xi x q^r/s} (\la z q^r/t;-\tfrac{uq^{-r}}{\la\xi} \sqrt{qst}, -\tfrac{q^{-r}}{u\la\xi} \sqrt{qst};q)
\frac{(-zq^{1+k}, zq^{1+k}/t;q)_\infty}{(\tfrac{z u^{\pm 1}}{\xi} \sqrt{qs/t};q)_\infty}\,d_q z \\
=\, &\, (1-q)|\la| q^{(m+\frac12)(n-k)} q^{-\frac12 r(r+1)} (\la \xi)^{-r}(\la/\xi)^m (st)^{r/2} (s/t)^{m/2} \\
&  \times \frac{ (q,xq^{1+n}/s, -xq^{1+n};q)_\infty \te(-\la^2 q/t;q) }{ (-\xi x q^{1+r}/s,\tfrac{\la u^{\pm 1} }{\xi} \sqrt{qs/t}, -\tfrac{u^{\pm 1} q^{-r}}{\la \xi} \sqrt{qst};q)_\infty }  \frac{ (-\xi q^{k-m+1}, \xi q^{k-m+1}/s;q)_m }{ (-\la q^{n-m+1}, \la q^{n-m+1}/t;q)_m}\\
&  \times p_n(\mu_u;-\la \sqrt{qs/t}, \la \sqrt{q/st}, \tfrac{x}{\la} \sqrt{qt/s};q)\, L_m(-\xi x q^{1+r}/s;-\xi q^{k-m}, \xi q^{k-m}/s;q).
\end{split}
\]
\end{thm}
\begin{proof}
This follows from identity \eqref{eq:sum Rcc=cc}, with $(q^{2\si},q^{2\tau}) \mapsto (s,t)$ and $q^2 \mapsto q$, by writing it in terms of the orthogonal polynomials/functions, and from careful bookkeeping. We also use the identity
\[
\frac{ (\tfrac{\xi u^{\pm 1} }{\la} \sqrt{ qt/s}, -\la u^{\pm 1} \sqrt{qs/t}, \la u^{\pm 1} \sqrt{q/st};q)_\infty }{ (-\xi u^{\pm 1} \sqrt{qt/s}, \xi u^{\pm 1} \sqrt{q/st};q)_\infty } = (\tfrac{\la u^{\pm 1}}{\xi} \sqrt{qs/t};q)_\infty,
\]
and \eqref{eq:handy identities}.
\end{proof}
Taking the limits $s\to 0$ and $s,t \to 0$, which corresponds to $\rho_{\tau,\si} \to \rho_{\tau,\infty}$ and $\rho_{\tau,\si} \to -\ga\ga^*$ in the algebra $\mathcal A_q$, we obtain the following results. The functions involved are big $q$-Laguerre polynomials $L_n$ defined by \eqref{eq:bigqLaguerre}, the Wall polynomials $p_n$ defined by \eqref{eq:defWallpol}, and the $q$-Bessel functions \eqref{eq:bigqBessel} and \eqref{eq:HahnExtonqBessel}.
\begin{cor} \label{cor:identities}Let $k,m,n,u\in \N$, $x\in \Z$ and $t>0$. The following identities hold:
\[
\begin{split}
\int_{-1}^{t} & L_k(zq^{1+u}/t; -z,z/t;q) L_m(zq^{1+k-m+n}/t;q^{n-m},-q^{n-m}/t;q) \\
& \times \mathcal J_{u+m-k-n,x-u}^{q/t}(-zq^{k-m+n}/t;q) \frac{ (-zq,zq/t;q)_\infty }{|z|^k (zq^{u+1}/t;q)_\infty}d_q z \\
=&\, (1-q)\,t^{1-m} q^{-m(m-1)-mn-n+m + x(k-m)}\\
& \times \frac{(-q/t,q^{k-m+1};q)_m (q,q^{x+1};q)_\infty \te(-qt;q) }{(-q^{n-m+1}/t,q^{n-m+1};q)_m (-q^{u+1}/t, q^{u+m-k-n+1};q)_\infty} \\
& \times  L_n(-q^{u+1}/t;-1/t,q^x;q) p_m(q^{n+x};q^{k-m};q).
\end{split}
\]
and
\[
\begin{split}
\sum_{z \in \N} q^{z(u-k+1)} \frac{(q^{m-k-n+u+1};q)_\infty}{(q;q)_z} J_{m-k-n+u}(q^{x+z};q) p_k(q^u;q^{z};q) p_m(q^{k+z};q^{n-m};q) \\
= q^{x(k-m)}\frac{ (q^{k-m+1};q)_m}{(q^{n-m+1};q)_m}(q^{x+1};q)_\infty p_n(q^u;q^x;q) p_m(q^{n+x};q^{k-m};q).
\end{split}
\]
\end{cor}
\begin{proof}
For the first identity substitute $(u,x,\xi,\la) \mapsto (-q^u\sqrt{q/st},-q^x,-1,-1)$ in Theorem \ref{thm:sum 2phi2 3phi2 3phi2}, and let $s\to0$. For the second identity we also substitute $z \mapsto -q^z$, and let $t\to 0$.
\end{proof}
\begin{remark} \*
\begin{enumerate}[(i)]
\item Corollary \ref{cor:identities} shows that the $(\tau,\infty)$-coupling coefficients, respectively the $(\infty,\infty)$-coupling coefficients, can be expressed in terms of big $q$-Bessel functions, respectively Hahn-Exton $q$-Bessel functions.

\item The second identity in Corollary \ref{cor:identities} is recently obtained in \cite{dCK} by De Commer and Koelink in a different quantum group setting; they use the quantum linking groupoid between quantum $\mathrm{SU}(2)$ and quantum $\mathrm{E}(2)$. They consider the identity as a $q$-analog of Erd\'elyi's identity \cite{Er} which roughly says that the Hankel transform maps a product of two Laguerre polynomials to the product of two Laguerre polynomials. The first identity in Corollary \ref{cor:identities} can be considered as another $q$-analog of Erd\'elyi's identity.
\end{enumerate}
\end{remark}

\appendix
\section{$q$-integral evaluations}

\subsection{Proof of Proposition \ref{prop:integralPQ}}
We prove the following result.
For $n \in \N$,
\begin{equation} \label{eq:intPn}
\begin{split}
\frac{1}{1-q}&\int_{z_-}^{z_+} P_n(x) \frac{ (qx/z_-,qx/z_+;q)_\infty}{(cx,dx;q)_\infty} d_qx \\
&= z_+(-d)^n q^{-\frac12n(n-1)} \frac{ (q,cdz_-z_+q^n;q)_\infty \te(z_-/z_+;q)}{ (cz_-q^n,dz_-,cz_+, dz_+;q)_\infty} \rphis{2}{1}{q^{-n}, dz_+ }{q^{1-n}/cz_-}{q, \frac{q}{dz_-}}\\
&=z_+(-d)^n q^{-\frac12n(n-1)} \frac{ (q,cdz_-z_+q^n;q)_\infty \te(z_-/z_+;q)}{ (cz_-,dz_-,cz_+q^n, dz_+;q)_\infty} \rphis{2}{1}{q^{-n}, dz_- }{q^{1-n}/cz_+}{q, \frac{q}{dz_+}},
\end{split}
\end{equation}
and
\begin{equation} \label{eq:intQn1}
\begin{split}
\frac{1}{1-q}&\int_{z_-}^{z_+} Q_n(x) \frac{ (qx/z_-,qx/z_+;q)_\infty}{(cx,dx;q)_\infty} d_qx \\
& =z_+(c/q)^n q^{-\frac12 n(n-1)} (q/cz_-;q)_n (q,q^{n+1};q)_\infty \te(z_-/z_+;q) \rphis{2}{1}{q^{-n}, q/dz_+}{cz_-q^{-n}}{q,dz_-}\\
& =z_+ (c/q)^n q^{-\frac12 n(n-1)} (q/cz_+;q)_n (q,q^{n+1};q)_\infty \te(z_-/z_+;q) \rphis{2}{1}{q^{-n}, q/dz_-}{cz_+q^{-n}}{q,dz_+}.
\end{split}
\end{equation}
Furthermore, for $n \in -\N_{\geq 1}$,
\begin{equation} \label{eq:intQn2}
\int_{z_-}^{z_+} Q_n(x) \frac{ (qx/z_-,qx/z_+;q)_\infty}{(cx,dx;q)_\infty} d_qx = 0.
\end{equation}
Here $P_n$ and $Q_n$ are defined by \eqref{eq:defPn} and \eqref{eq:defQn}.\\

The functions $P_n$ and $Q_n$ are essentially two instances of the same function. Indeed, we define
\[
\psi_\ga(x;c,d) = (cq\ga x,dx;q)_\infty \rphis{1}{1}{q\ga}{cq\ga x}{\frac{cq}{d}},
\]
and
\begin{equation} \label{eq:defphigamma}
\varphi_\ga(x) = \frac{ \te(cz_+,dz_+\ga;q)}{\te(d/c;q)} \psi_\ga(x;c,d) + (c \leftrightarrow d),
\end{equation}
where $(c\leftrightarrow d)$ means that the preceding expression is repeated, but with $c$ and $d$ interchanged. Now $P_n$ and $Q_n$ are related to $\varphi_\ga$ by
\begin{equation} \label{eq:phi=PQ}
\varphi_{\ga}(x) =
\begin{cases}
\displaystyle (cdz_+)^{-n} (q^{n+1};q)_\infty \te(cz_+,dz_+;q)\, P_n(x),& \text{if}\ \ga = q^n,\\ \\
\displaystyle (z_-)^n (q^{n+2}/cdz_-z_+;q)_\infty \, Q_n(x), & \text{if}\ \ga = \dfrac{q^{n+1}}{cdz_-z_+},
\end{cases}
\end{equation}
see \cite[\S3.4]{Groen}.

To prove identities \eqref{eq:intPn},\eqref{eq:intQn1} and \eqref{eq:intQn2} we evaluate the $q$-integral
\[
J_\ga(c,d) = \int_{z_-}^{z_+} \varphi_\ga(x) \frac{ (qx/z_-,qx/z_+;q)_\infty}{(cx,dx;q)_\infty} d_qx.
\]
We will use the evaluation formula \cite[(2.10.20)]{GR}
\begin{equation} \label{eq:qintegraleval}
\int_{A}^{B} \frac{ (qx/A,qx/B;q)_\infty }{ (Cx,Dx;q)_\infty } d_q x = (1-q)B \frac{ (q,ABCD;q)_\infty \te(A/B;q) }{ (AC,AD,BC,BD;q)_\infty},
\end{equation}
which is the nonterminating $q$-Vandermonde sum written as a $q$-integral.
By \eqref{eq:defphigamma} the $q$-integral $J_\ga$ splits as
\[
J_\ga(c,d) =  \frac{ \te(cz_+,dz_+\ga;q)}{\te(d/c;q)} I_\ga(c,d) + (c \leftrightarrow d),
\]
where
\[
I_\ga(c,d) = \int_{z_-}^{z_+} \psi_\ga(x;c,d) \frac{ (qx/z_-,qx/z_+;q)_\infty}{(cx,dx;q)_\infty} d_qx.
\]
To evaluate $I_\ga$ we write $\psi_\ga$ as
\[
\psi_\ga(x;c,d) = (dx,q\ga, cq/d;q)_\infty \rphis{2}{1}{0,cx}{cq/d}{q\ga}, \qquad |\ga|<q^{-1},
\]
see \cite[(2.3)]{Groen}. We interchange the order of summation and $q$-integration, which is allowed by absolute convergence, then from \eqref{eq:qintegraleval} we obtain
\begin{equation} \label{eq:I=2phi1}
\begin{split}
I_\ga(c,d) &= (q\ga,cq/d;q)_\infty \sum_{n=0}^\infty \frac{ (q\ga)^n}{(q,cq/d;q)_n } \int_{z_-}^{z_+} \frac{ (qx/z_-,qx/z_+;q)_\infty }{(cq^n x;q)_\infty} d_qx \\
& = (1-q)z_+ \frac{ (q,q\ga, cq/d;q)_\infty \te(z_-/z_+;q) }{(cz_-,cz_+;q)_\infty} \rphis{2}{1}{cz_-, cz_+}{cq/d}{q,q\ga},
\end{split}
\end{equation}
provided $|\ga|< q^{-1}$. So we have
\begin{equation} \label{eq:Jgamma=2phi1+2phi1}
J_\ga(c,d) = K_\ga\, \frac{ (q/cz_+;q)_\infty \te(dz_+\ga;q)}{(cz_-,d/c;q)_\infty } \rphis{2}{1}{cz_-,cz_+}{cq/d}{q,q\ga} + (c \leftrightarrow d), \qquad |\ga|<q^{-1},
\end{equation}
where $K_\ga\, = (1-q)z_+ (q,q\ga;q)_\infty \te(z_-/z_+;q)$. For any $x \in z_-q^\Z \cup z_+q^\Z$ the function $\ga\mapsto \psi_\ga(x;c,d)$ is an entire function, so $\ga \mapsto \varphi_\ga(x)$ is analytic on $\C\setminus\{0\}$. Therefore, by absolute convergence $\ga\mapsto J_\ga(c,d)$ is also analytic on $\C\setminus\{0\}$. An analytic continuation of the right hand side of \eqref{eq:Jgamma=2phi1+2phi1} in $\ga$ can be obtained by applying the three-term transformation formula for $_2\varphi_1$-functions \cite[(III.32)]{GR} with parameters $(A,B,C,Z) = (cz_-,dz_-,qz_-/z_+,q/cdz_-z_+\ga)$, which gives us
\begin{equation} \label{eq:Jgamma}
J_\ga(c,d) = K_\ga \frac{ (qz_-/z_+;q)_\infty \te(cdz_-z_+\ga;q) }{(cz_-,dz_-;q)_\infty } \rphis{2}{1}{cz_-,dz_-}{qz_-/z_+}{q,\frac{q}{cdz_-z_+\ga}}, \qquad |\ga|>\frac{q}{|cdz_-z_+|}.
\end{equation}
Applying Heine's transformation \cite[(III.2)]{GR} in \eqref{eq:Jgamma} gives
\[
J_\ga(c,d) = K_\ga \frac{ (q/dz_+, q/cz_+\ga, cdz_-z_+\ga;q)_\infty}{(cz_-,dz_-;q)_\infty} \rphis{2}{1}{1/\ga, dz_-}{q/cz_+\ga}{q,\frac{ q}{dz_+}}, \qquad \left|q/dz_+\right|<1.
\]
Setting $\ga = q^n$, we see that the $_2\varphi_1$-series terminates, so that the condition $|q/dz_+|<1$ is not needed. Using \eqref{eq:phi=PQ} this gives us the second identity in \eqref{eq:intPn}. The first one is obtained by interchanging $z_-$ and $z_+$, or, alternatively, reversing the order of summation in the $_2\varphi_1$-series and using $c \leftrightarrow d$ symmetry.

The results involving $Q_n$ follow from applying Heine's transformation \cite[(III.1)]{GR} in \eqref{eq:Jgamma};
\[
J_\ga(c,d) = K_\ga(z_-,z_+) \frac{ (cdz_-z_+\ga, q/dz_+\ga;q)_\infty }{(cz_-;q)_\infty } \rphis{2}{1}{q/dz_+, q/cdz_-z_+\ga}{q/dz_+\ga}{q,dz_-}, \qquad |dz_-|<1.
\]
We set $\ga= q^{n+1}/cdz_-z_+$ with $n \in\Z$. In case $n \in \N$, the $_2\varphi_1$-series is a finite sum, and then the condition $|dz_-|<1$ is not needed. Using \eqref{eq:phi=PQ} we obtain the first expressions in \eqref{eq:intQn1}. Interchanging $z_-$ and $z_+$ gives the other expression. In case $n \in -\N_{\geq 1}$, $J_\ga(c,d)$ vanishes, since $(cdz_-z_+\ga;q)_\infty =0$, which proves \eqref{eq:intQn2}.

\subsection{Proof of Proposition \ref{prop:integral qMeixner}} \label{app:integral qMeixner}
We prove
\[
\begin{split}
\int_{-1}^{t} &\phi_\ga(xq^m;aq^{-m}, bq^{-m};q) \frac{ (-xq,xq/t;q)_\infty }{(-ax,-bx;q)_\infty }\, d_q x   = \\ &(1-q)t \frac{ (q,abt\ga q^{-m};q)_\infty \te(-1/t;q) }{ (aq^{-m},bq^{-m},-at,-bt;q)_\infty }  (-q^{-m}/t;q)_m \rphis{3}{2}{ q^{-m},-bt,-at}{abt\ga q^{-m}, -qt}{q,q},
\end{split}
\]
where $\phi_\ga$ denotes the $q$-Meixner function \eqref{def:qMeixner}. \\

We write the $q$-Meixner function as a $_2\varphi_1$-series;
\[
\phi_\ga(xq^m;aq^{-m}, bq^{-m};q) = \frac{ (-a\ga q^{-m};q)_\infty }{(aq^{-m};q)_\infty } \rphis{2}{1}{-1/\ga, -bx}{bq^{-m}}{q,-a\ga q^{-m}}, \qquad |a\ga q^{-m}|<1.
\]
By interchanging the order of summation and $q$-integration, and using \eqref{eq:qintegraleval}, we obtain
\[
\begin{split}
\int_{-1}^{t} & \phi_\ga(xq^m;aq^{-m}, bq^{-m};q) \frac{ (-xq,xq/t;q)_\infty }{(-ax,-bx;q)_\infty }\, d_q x \\
 &= \frac{ (-a\ga q^{-m};q)_\infty }{(aq^{-m};q)_\infty } \sum_{n=0}^\infty \frac{ (-1/\ga;q)_n (-a\ga q^{-m})^n }{(q,bq^{-m};q)_n } \int_{-1}^{t} \frac{ (-xq,xq/t;q)_\infty }{(-ax,-bxq^n;q)_\infty }\, d_q x \\
&= \frac{ (-a\ga q^{-m};q)_\infty }{(aq^{-m};q)_\infty } \sum_{n=0}^\infty \frac{ (-1/\ga;q)_n (-a\ga q^{-m})^n }{(q,bq^{-m};q)_n } (1-q)t \frac{ (q,-abtq^n;q)_\infty \te(-1/t;q) }{(a, bq^n, -at, -btq^n;q)_\infty} \\
& = (1-q)t \frac{ (q,-a\ga q^{-m},-abt;q)_\infty \te(-1/t;q) }{ (a,aq^{-m},b,-at,-bt;q)_\infty } \rphis{3}{2}{ -1/\ga,b,-bt}{bq^{-m},-abt}{q,-a\ga q^{-m}} .
\end{split}
\]
We apply transformation formulas \cite[(III.9)]{GR} with parameters $(A,B,C,D,E)=(-bt,b,-1/\ga,bq^{-m},-abt)$, and \cite[(III.13)]{GR} with $(B,C,D,E) = (-b\ga q^{-m},-bt,abt\ga q^{-m},bq^{-m})$, then we find
\[
\begin{split}
\rphis{3}{2}{ -1/\ga,b,-bt}{bq^{-m},-abt}{q,-a\ga q^{-m}} &= \frac{ (a,abt\ga q^{-m};q)_\infty }{ (-abt,-a\ga q^{-m};q)_\infty} \rphis{3}{2}{q^{-m}, -bt, -b\ga q^{-m} }{bq^{-m}, abt\ga q^{-m}}{q, a}\\
& = \frac{ (a,abt\ga q^{-m};q)_\infty (-q^{-m}/t;q)_m}{ (-abt,-a\ga q^{-m};q)_\infty (bq^{-m};q)_m}  \rphis{3}{2}{ q^{-m},-at,-bt}{abt\ga q^{-m}, -qt}{q,q}.
\end{split}
\]
This proves the result for $|a\ga q^{-m}|<1$. By analytic continuation it holds for all $\ga$.

\end{document}